\documentclass[11pt]{article}

\usepackage{amsmath, amsthm}
\usepackage{amsfonts}
\usepackage{amssymb}
\usepackage[usenames,dvipsnames]{color}
\usepackage{graphicx}
\usepackage{hyperref}
\usepackage{bm}
\usepackage{bbm}
\usepackage{listings}
\usepackage{geometry}
\usepackage{lipsum} 

\usepackage{ulem}

\usepackage{color,calc}
\definecolor{junglegreen}{rgb}{0.16, 0.67, 0.53}

\makeatletter
\definecolor{shade}{gray}{0.8}
        {
          \raggedright
        \setlength{\rightmargin}{\leftmargin}
        \setlength{\itemsep}{-12pt}
        \setlength{\parsep}{20pt}
        \begin{lrbox}{\@tempboxa}%
        \begin{minipage}{\linewidth-2\fboxsep}
        }%
        {
        \end{minipage}%
        \end{lrbox}%
        \fcolorbox{black}{shade}{\usebox{\@tempboxa}}\newline
        }%
\makeatother

\newtheorem{theorem}{Theorem}

\newtheorem{corollary}{Corollary}
\newtheorem{assumption}{Assumption}
\newtheorem{definition}{Definition}
\newtheorem{remark}{Remark}

\newcommand{\iu}{\mathrm{i}} 

\newcommand{\dint}{\displaystyle\int}

\renewcommand{\eqref}[1]{\hyperref[#1]{(\ref*{#1})}}

\newcommand{\dd}{\mathrm{d}}

\renewcommand{\dd}{{\rm d}}

\newcommand*{\norm}[1]{\lVert #1 \rVert}

\newcommand{\h}{\mathcal{H}}
\newcommand{\s}{\mathcal{S}}

\renewcommand{\P}{\mathbb{P}}

\newcommand{\re}{\mathbb{R}}
\newcommand{\bP}{\mathbf{P}}

\newcommand{\bE}{\mathbf{E}}

\newcommand*{\pref}[1]{\hyperref[#1]{(\ref*{#1})}}
\newcommand*{\refpref}[2]{\hyperref[#2]{\ref*{#1}(\ref*{#2})}}

\newcommand{\R}{\mathbb{R}}

\allowdisplaybreaks 
\newcommand{\one}[1]{\mathbbm{1}_{\{#1\}}}

\newcommand{\ubar}[1]{\underline{#1}}

\newcommand{\dE}{\hat{\textbf{E}}}
\newcommand{\ordref}[1]{\xi_#1 - \ubar{\xi}_#1}
\newcommand{\minmod}[1]{\ubar{\Theta}_#1}
\definecolor{amethyst}{rgb}{0.6, 0.4, 0.8}
\definecolor{shitbrown}{rgb}{0.55, 0.71, 0.0}
\definecolor{aqua}{rgb}{0.0, 1.0, 1.0}
\definecolor{asparagus}{rgb}{0.53, 0.66, 0.42}
\definecolor{amber(sae/ece)}{rgb}{1.0, 0.49, 0.0}
 	\definecolor{armygreen}{rgb}{0.29, 0.33, 0.13}
	\definecolor{shitbrown}{rgb}{0.43, 0.21, 0.1}
	\definecolor{brightpink}{rgb}{1.0, 0.0, 0.5}
	\definecolor{brightube}{rgb}{0.82, 0.62, 0.91}
	 	\definecolor{byzantine}{rgb}{0.74, 0.2, 0.64}
		\definecolor{chartreuse(web)}{rgb}{0.5, 1.0, 0.0}

\title{Williams' path decomposition for 
self-similar \\Markov processes in $\mathbb{R}^d$}

\author{A. E. Kyprianou\footnote{Department of Statistics, University of Warwick, Coventry
CV4 7AL, UK. Email: \texttt{andreas.kyprianou@warwick.ac.uk} }  \
 M.  Motala\footnote{Department of Mathematical Sciences, University of Bath, Claverton Down, Bath, BA2 7AY, UK. Email: \texttt{mm2980@bath.ac.uk}}, and V. Rivero\footnote{
 Centro de Investigaci\'on
 en Mathem\'aticas, A. C. 
 Calle Jalisco s/n. 36240
M\'exico.
E-mail: \texttt{rivero@cimat.mx}}
}

\begin{document}
\maketitle
\begin{abstract}The classical result of Williams \cite{Williams} states that a Brownian motion with positive  drift $\mu$  and issued from the origin is equal in law to a Brownian motion with unit negative drift, $-\mu$, run until it hits a negative threshold, whose depth below the origin is independently and exponentially distributed with parameter $2\mu$, after which it behaves like a Brownian motion conditioned never to go below the aforesaid threshold (i.e. a Bessel-3 process, or equivalently a Brownian motion conditioned to stay positive, relative to the threshold). In this article  we consider the analogue of Williams' path decomposition for a general self-similar Markov process (ssMp) on $\mathbb{R}^d$. Roughly speaking, we will prove that law of a ssMp, say $X$, in $\mathbb{R}^d$ is equivalent in law to the  concatenation of paths described as follows: suppose that 
we sample the  point $x^*$  according to the  law of the point of closest reach to the origin, sample; given $x^*$, we build $X^{\downarrow}$ having the law of $X$ conditioned to hit $x^*$ continuously without entering the ball of radius $|x^*|$; then, we construct $X^\uparrow$ to have the law of $X$ issued from $x^*$ conditioned never to enter the ball of radius $|x^*|$; glueing the path of $X^\uparrow$ end-to-end with $X^\downarrow$ via the point $x^*$ produces a process which is equal in law to our original ssMp $X$. 
In essence, Williams' path decomposition in the setting of a ssMp follows directly from an analogous decomposition for Markov additive processes (MAPs). The latter class are intimately related to the former via a space-time transform known as the Lamperti--Kiu transform. As a key feature of our proof of Williams' path decomposition, will prove the analogue of Silverstein's duality identity for the excursion occupation measure, cf. \cite{silverstein}, for general Markov additive processes (MAPs). \\

\noindent {\bf Key words:} Self-similar Markov processes, Path decompositions, fluctuation theory for Markov additive processes.\\
\noindent {\bf Mathematics Subject Classification:}  60J80, 60E10.
\end{abstract}

\section{Introduction} 
In this article we explore a number of fundamental path decompositions for a general family of  self-similar Markov processes in $\mathbb{R}^d$ through an intimate connection to Markov additive processes.

\smallskip

We recall that a classical result of Williams \cite{Williams} states that a Brownian motion with positive unit drift is equal in law to a Brownian motion with unit negative drift run until it hits a negative threshold, whose depth below the origin is independently and exponentially distributed with parameter 2, after which it behaves like a Brownian motion conditioned never to go below the aforesaid threshold (i.e. a Bessel-3 process relative to the threshold).  Williams' path decomposition is, in essence, a more detailed version of a result  of Millar \cite{millar1978} which states that, Markov processes in a rather general class, can always be split at its minimum, resulting in a conditional independence of the pre and post infimum process. 

\smallskip

It turns out that Williams' path decomposition generalises nicely to the setting of self-similar Markov processes as soon as we can understand two types of conditionings in the larger setting of ssMps which were previously discussed in the setting of positive self-similar Markov processes in~\cite{chaumont-rivero} and for isotropic stable L\'evy processes in \cite{KPS}. The first is to condition the process to continuously approach a patch on a pre-specified sphere. The second is to condition the process to remain outside of a pre-specified sphere when starting on its surface. 

\smallskip

Roughly speaking, the analogue of Williams' path decomposition we will prove states the law of a self-similar Markov process, say $X$, in $\mathbb{R}^d,$ is equivalent in law to the  concatenation of paths described as follows. Suppose that    we  sample  a point $x^*$  according to the  law of the point of closest reach to the origin. That is to say $x^*$ is a sample from the random variable $X^* : = X_{G(\infty)}$, where $G(\infty) = \sup\{t>0: \norm{X_t} = \inf_{s\geq0}\norm{X_s}\}$. Given $x^*$, we build $X^{\downarrow}$ having the law of $X$ conditioned to hit $x^*$ continuously without entering the ball of radius $|x^*|$. Moreover, we construct $X^\uparrow$ to have the law of $X$ issued from $x^*$ conditioned never to enter the ball of radius $|x^*|$. Glueing the path of $X^\uparrow$ end-to-end with $X^\downarrow,$ via the point $x^*,$ produces a process which is equal in law to our original ssMp $X$.
\smallskip

{Before we can formally state our results,  we must first  describe the two fundamental classes of processes we deal with and how they are connected via a space time transformation, then introduce some assumptions that will be in force for our main results to hold. That is the purpose of the next section.  
\section{Main assumptions}}
\subsection{Self-similar Markov processes (ssMp)}
Suppose $\h$ is a locally compact
{subset}
of $\R^{d}\setminus\{0\}$ ($d\ge 1$). An $\h$-valued self-similar Markov process (ssMp for short) $(X,\P)=((X_t)_{t \geq 0},\{\P_{z}:z\in\h\})$ is an $\h$-valued c\`adl\`ag Markov process killed at $0$ with $\P_{z}\left(X_{0}=z\right)=1$, which fulfils the scaling property, namely, there exists an $\alpha>0$ such that for any $c>0$,
$$((cX_{c^{-\alpha}t})_{t \geq 0},\P_z) \text{ has the same law as }((X_{t}, t\ge 0),\P_{cz}), \quad \forall z \in \h.$$
It follows from the scaling property that $\h=c\h$ for all $c>0$. Therefore $\h$ is necessarily a cone of $\R^{d}\setminus\{0\}$ which has the form
$$\h =\phi(\R\times \s),$$
where $\s$ is a locally compact
{subset}
of $\mathbb{S}^{d-1}$ and $\phi$ is the homeomorphism from $\R\times \mathbb{S}^{d-1}$ to $\R^{d}\setminus\{0\}$ defined by $\phi(y,\theta)=\theta\mathrm{e}^{y}$.

\smallskip

A fundamental tool in the study of ssMp is the Lamperti-Kiu transform which links the aforementioned class to the family of Markov additive processes. Let us introduce the latter before describing the nature of the Lamperti-Kiu transform. 

\subsection{Markov additive processes (MAPs)} 
Suppose $(\xi_{t},\Theta_{t}, t\ge 0)$ is the coordinate process in $\mathbb{D}({\R\times\s})$ and
$$\left((\xi,\Theta),\bP \right)=\left((\xi_{t},\Theta_{t}, t\ge 0),\mathcal{F}_{\infty},(\mathcal{F}_{t}, t\ge 0),\{\bP_{x,\theta}:(x,\theta)\in\R\times\s\}\right),$$ is a (possibly killed) Markov process with $\bP_{x,\theta}\left(\xi_{0}=x,\Theta_{0}
=\theta\right)=1$. Here $(\mathcal{F}_{t}, t\ge 0)$ is the minimal augmented and right continuous admissible filtration, and $\mathcal{F}_{\infty}=\bigvee_{t=0}^{+\infty}\mathcal{F}_{t}$.

  \begin{definition}
    The process $\left((\xi,\Theta),\bP\right)$ is called a Markov additive process (MAP) on $\R\times\s$ if, for any $t \geq 0$, given $\{(\xi_s,\Theta_s), s\leq t\}$, the process $(\xi_{s+t}-\xi_t,\Theta_{s+t})_{s\geq 0}$ has the same law as $(\xi_{s},\Theta_{s}, s\ge 0)$ under $\bP_{0,v}$ with $v=\Theta_t$. We call $\left((\xi,\Theta),\bP\right)$ a nondecreasing MAP if $\xi$ is a nondecreasing process on $\R$.
  \end{definition}

For a MAP process $\left((\xi,\Theta),\bP\right)$, we call $\xi$ the \textit{ordinate} and $\Theta$ the \textit{modulator}. By definition we can see that a MAP is conditionally translation invariant in $\xi$ in the sense that $((\xi_{t},\Theta_{t}, t\ge 0),\bP_{x,\theta})$ is equal in law to $((\xi_{t}+x,\Theta_{t}, t\ge 0),\bP_{0,\theta})$ for all $x\in \R$ and $\theta\in\s $. (The `conditional' translation is dependent on the current value of the modulator.)  We assume throughout the paper the following path regularity assumptions for the process $(\xi,\Theta)$.

\begin{assumption}\label{ass1}
The process $(\Theta_{t}, t\ge 0)$ is a Hunt process with invariant probability distribution $\pi$ on $\mathcal{S}$ and $(\xi_{t}, t\ge 0)$ is quasi-left continuous on $[0,\zeta)$. 
\end{assumption}

\smallskip

Whilst MAPs have found a prominent role in e.g. classical applied probability models for queues and dams, c.f. \cite{AsmussenQueue} when $\Theta$ is a Markov chain, the case that $\Theta$ is a general Markov process has received somewhat less attention. Nonetheless a core base of literature exists in the general setting from the 1970s and 1980s thanks to e.g. \cite{AsmussenQueue,AA,CPR,Cinlar1, CinlarI&II, CinlarMRT, KPal, Kaspi1,Kaspi2}. Sufficient conditions in terms of $(\xi,\Theta)$ for $\pi$ to exist are given in \cite{TAMS}.

\smallskip

Now that we have defined what we mean by a MAP and introduced some of its key properties,  we can return to the setting of self-similar Markov processes and describe the promised connection between the two classes of processes. 

\subsection{The Lamperti-Kiu representation}

 Suppose first that $(X,\P_z)$ is an $\h$-valued ssMp started at $z\in\h$ with index $\alpha>0$ and lifetime $\varsigma$, then there exists a Markov additive process  $(\xi,\Theta)$ on $\R \times \s$ started at $(\log \|z\|,\arg(z))$ with lifetime $\zeta$ such that
  \begin{equation}\label{eq:lamperti_kiu}
   X_t  = \exp\{\xi_{\varphi(t)}\} \Theta_{\varphi(t)}\one{t < \varsigma}, \qquad\forall t \geq 0,
  \end{equation}
   where $\varphi(t)$ is the time-change defined by
  \begin{equation}\label{eq:time_change}
   \varphi(t):=\inf\left\{s>0:\int_0^s \exp\{\alpha\xi_u\}\,{\rm d}u > t\right\},
  \end{equation}
  and $\zeta=\int_{0}^{\varsigma}\|X_{s}\|^{-\alpha}{\rm d}s$.
 We denote the law of $(\xi,\Theta)$ started from $(y,\theta)\in \R\times \s$ by $\bP_{y,\theta}$.
  Conversely given a MAP $(\xi,\Theta)$ under $\bP_{y,\theta}$ with lifetime $\zeta$, the process $X$ defined by~\eqref{eq:lamperti_kiu} is a ssMp started from $z=\theta {\rm e}^y$ with lifetime $\varsigma=\int_{0}^{\zeta}\mathrm{e}^{\alpha \xi_{s}}{\rm d}s$. 
 
\subsection{Duality}\label{dualsection}
We refer to \cite{TAMS} for the summary of duality we present here. The key notion of duality for MAPs requires the existence of the invariant distribution $\pi$ as per Assumption \ref{ass1}. 

\begin{assumption}\label{ass2} There exists a Markov additive process $((\xi,\Theta),\widetilde{\bP}) $ with probabilities $\widetilde{\bP} = (\widetilde{\bP}_{y, \theta}, y\in\mathbb{R}, \theta\in \mathcal{S})$, 
   which is linked to $((\xi,\Theta),\bP)$ through the  weak reversibility property
  \begin{equation}\label{condi:weak reversability}
      \bP_{0,\theta}(\xi_t \in {\rm d}z; \Theta_t \in {\rm d}v)\pi({\rm d}\theta) = \widetilde{\bP}_{0,v}(\xi_t \in {\rm d}z; \Theta_t \in {\rm d}\theta)\pi({\rm d}v), \quad \forall t \geq 0.
  \end{equation}
\end{assumption}  
  
Observe that by integrating \eqref{condi:weak reversability} over variable $z$, it follows that the Markov processes $\Theta$ with probabilities $(\bP_{0,\theta},\theta\in\s)$ and  $\Theta$ with probabilities $(\widetilde{\bP}_{0,\theta},\theta\in\s)$ are in weak duality with respect to the measure $\pi$. Hereafter we denote by $\hat{\bP}_{x,\theta}$ the law of $(-\xi,\Theta)$ under $\widetilde{\bP}_{-x,\theta}$. We will consistently  use the `hat' notation to specify the mathematical quantities related to the process $((\xi,\Theta),\hat{\bP})$.

It has been proved in Lemma 3.8 in \cite{TAMS} that the processes $((\xi,\Theta),\bP)$ and $((\xi,\Theta),\hat\bP)$ are in weak duality with respect to the measure ${\rm Leb} \otimes \pi$, where ${\rm Leb}$ is the Lebesgue measure on $\R$. As usual, weak duality is closely related to time reversal, having as one of its possible expressions the following, which is the the generalization of a classical and fundamental result for random walks and Lévy processes. For $t>0$ arbitrarily fixed, we have that $ \forall x\in\R$
\begin{equation}\label{lem:easy switch}
\left[\left((\xi_{(t-s)-}-\xi_{t},\Theta_{(t-s)-}),\ 0\leq s\leq t \right),\ \bP_{x,\pi} \right] \stackrel{\text{Law}}= \left[\left((\xi_{s},\Theta_{s}),\ 0\leq s\leq t\right),\ \hat{\bP}_{0,\pi}\right]. 
\end{equation}

The regularity assumption, Assumption \ref{ass3}, for $((\xi,\Theta),\bP),$ implies that, almost surely, the local minima of $\xi$ during a finite time interval are all distinct. In view of this and the above equality in law, we have the following classical result. For any fixed $t>0$, the vector $(\Theta_{0},t-\ubar{g}_{t},\Theta_{t},\ordref{t},\ubar{g}_{t},\ubar{\Theta}_{t},-\ubar{\xi}_{t})$ under $\hat{\bP}_{0,\pi},$
has the same distribution as $(\Theta_{t},\ubar{g}_{t},\Theta_{0},-\ubar{\xi}_{t},t-\ubar{g}_{t},\ubar{\Theta}_{t},\ordref{t}),$ under $\bP_{0,\pi}$. Thus, for any bounded measurable test function $H$, we have 
\begin{equation}\label{prop:equal dist}
\bE_{0,\pi}[H(\Theta_t, \ubar{g}_t, \Theta_0, -\ubar{\xi}_t,t - \ubar{g}_t, \minmod{t}, \ordref{t})  )] = \hat{\mathbf E}_{0,\pi}[ H(\Theta_0,  t - \ubar{g}_t, \Theta_t, \ordref{t},  \ubar{g}_t, \minmod{t}, -\ubar{\xi}_t ) ]. 
\end{equation} See \cite{TAMS} Proposition 3.9.

\subsection{MAP excursions from the ordinate minimum}
   
Let $\underline{\xi}_{t}:=\inf_{s\le t}\xi_{s}$ and $U_{t}:=\xi_{t}-\underline{\xi}_{t}$. Assumption \ref{ass1} implies that under $\bP_{0,\theta}$ the process $(\Theta_{t},\xi_{t},U_{t}, t\ge 0)$ is an $\s \times \R\times \R^{+}$-valued strong Markov process with right continuous and quasi-left continuos paths, starting at $(\theta,0,0)$, whose transition semigroup on $(0,+\infty)$ is given by
\[
{\mathtt P}_{t}f(v,x,u):=\bE_{0,v}\left[f\left(\Theta_{t},\xi_{t}+x,(u\vee \underline{\xi}_{t})-\xi_{t}\right)\right]\qquad t\geq0,
\]
for  nonnegative measurable function $f:\s \times\R \times\R^{+}\to \R^{+}$. 
We shall work with the canonical realization of $(\Theta_{t},\xi_{t},U_{t}, t\ge 0)$ on the sample space
{$\mathbb{D}({\s\times \R\times \R^{+}})$} (the space of c\`adl\`ag paths on $\s\times \R\times \R^{+}$).

\smallskip

We define
$\underline{M}:=\{t\ge 0 : U_{t}=0\}$ and $\underline{M}^{cl}$ its closure in $\R^{+}$.
Obviously the set $\R^{+}\setminus \underline{M}^{cl}$ is an open set and can be written as a union of intervals. We use $\underline{G}$ and $\underline{D}$, respectively, to denote the sets of left and right end points of such intervals.

\smallskip

Suppose the set $\mathbb{R}^{+}\setminus \underline{M}^{cl}$ is written as a union of random intervals $(g,d)\in \underline{G}\times \underline{D}$. For such intervals, define
\begin{equation} \label{excursionproc}
(\epsilon^{(g)}_{s},\Theta_s^{\epsilon^{(g)}}):=\begin{cases}
        (U_{g+s},\Theta_{g+s}),
         \quad &\hbox{if } 0\le s<d-g, \smallskip \\
        (U_{d},\Theta_{d}),
         \quad &\hbox{if } s\ge d-g.
        \end{cases}
\end{equation}
$(\epsilon^{(g)}_{s},\Theta_s^{\epsilon^{(g)}}, s\ge 0)$ is called an excursion from the minimum and $\zeta^{(g)}:=d-g$ is called its lifetime.
We use $\mathcal{E}$ to denote the collection
$\{(\epsilon^{(g)}_{s}(\omega),\Theta_s^{\epsilon^{(g)}}(\omega), s\ge 0):g\in \underline{G}(\omega),\ \omega\in
{\mathbb{D}({\s\times\R\times \R^{+}}})
\}$,
and call it the space of excursions. Furthermore, we denote by $(\epsilon^{(g)},\Theta^{\epsilon^{(g)}})$ the generic excursion process.

\smallskip

For any $y\in\R$, let $\tau^{-}_{y}:=\inf\{t>0:\xi_{t}<y\}$.
We say that $((\xi,\Theta),\bP)$ is
downwards regular
if
$$\bP_{0,\theta}\left(\tau^{-}_{0}=0\right)=1,\quad\forall \theta\in\s.$$

\begin{assumption}\label{ass3}
\noindent The process $\xi$ is downwards regular. 
\end{assumption}

\smallskip

This assumption implies that every point in $\s$ is regular for $\underline{M},$
{in the sense that $\bP_{0,\theta}\left(\underline{R}=0\right)=1$ for all $\theta\in\s$,}
where  $\underline{R}:=\inf\{t>0:\ t\in \underline{M}^{cl}\}$.
Thus by
{\cite[Theorem (4.1)]{maisonneuve}}
there exist a continuous additive functional $t\mapsto\underline{L}_{t}$ of
{$(\Theta_{t},\xi_{t},U_{t}, t\ge 0)$}
which is carried by $\s \times \R\times\{0\}$ and 
a kernel  ${\mathtt n}^{-}_{\theta}$ on $\mathcal{E}$ such  that
for any
{nonnegative}
measurable functionals
{$F:\mathbb{D}({\R^{+}\times \s })\to \R^{+}$}
and
{$G:\mathbb{R}^{+}\times \mathbb{D}({\mathbb{R}\times \s })\to\R^{+}$,} such that the process $u\mapsto G\big(u,(\xi_t,\Theta_t)_{t \leq u}\big)$ is left continuous,  
\begin{align}\label{eq:last exit}
   \bE_{y,\theta}&\left[ \sum_{g\in \underline{G}} G\big(g,(\xi_t,\Theta_t)_{t \leq g}\big) F\big(\epsilon^{(g)},{\Theta}^{\epsilon^{(g)}}\big) \right]\nonumber\\
   &=\bE_{y,\theta} \left[ \int_0^\infty {\rm d}\underline{L}_s\, G\big(s,(\xi_t,\Theta_t)_{t\leq s}\big) \int_{\mathcal{E}} {\mathtt n}^{-}_{\Theta_s}({\rm d}\epsilon,{\rm d}{\nu}) F(\epsilon,{\nu})  \right].
  \end{align}
  Equation \eqref{eq:last exit} goes by the name of the {\it last exit formula}.
We call the family  of measures $({\mathtt n}^{-}_{\theta}:\theta\in \s)$ the \textit{excursion measures at the minimum}.
{Moreover, by \cite[Theorem (5.1)]{maisonneuve}, ${\mathtt n}^-_{\theta}\left(\epsilon_{0}\neq0, \Theta^{\epsilon}_{0}\in \s\right)=0$, and under ${\mathtt n}^-_{\theta}$, the coordinate process, which is a generic excursion, is a Markov process with the same transition semigroup as $(\Theta, U)$ killed at its first hitting time of zero, see e.g.  \cite[(5.2)]{maisonneuve}.} The pair $(\ubar{L}, {\mathtt n}^-)$ is called an exit system. Note that in Maisonneuve's original formulation  $(\ubar{L}$, ${\mathtt n}^-)$ is not unique, but once $\ubar{L}$ is chosen, the measures $({\mathtt n}^{-}_{\theta}:\theta\in \s)$ exist and are unique up to $\ubar{L}$-negligible sets i.e. sets $\mathcal{A}$ such that $\bE_{x,\theta} \left( \int_0^\infty \one{ \Theta_s \in \mathcal{A}} \dd \ubar{L}_{s} \right)  = 0.$

\smallskip

Let $\underline{L}^{-1}$ be the right continuous inverse process of $\underline{L}$. Define $\xi^{-}_{t}:=-\xi_{\underline{L}^{-1}_{t}}$ and $\Theta^{-}_{t}:=\Theta_{\underline{L}^{-1}_{t}}$ for all $t$ such that $\underline{L}^{-1}_{t}<+\infty,$ and otherwise $\xi^{-}_{t}$ and $\Theta^{-}_{t}$ are both assigned to be the cemetery state $\partial$.
{One can verify by the strong Markov property that $(\underline{L}^{-1}_{t},\xi^{-}_{t},\Theta^{-}_{t}, t\ge 0)$ defines a MAP, whose first two elements are ordinates. Similarly, both $(\xi^{-}_{t},\Theta^{-}_{t}, t\ge 0) $ and $(\underline{L}^{-1}_{t},\Theta^{-}_{t}, t\ge 0) $ are MAPs.}
These three processes are referred to as  \textit{the descending ladder process}, \textit{the descending ladder height process} and \textit{the descending ladder time process}, respectively.

\smallskip

For later reference, we note from \cite{TAMS}, the MAP $\underline{L}^{-1} := (\underline{L}^{-1}_{t},\Theta^{-}_{t}, t\ge 0),$ on the one hand,  has  family of jump rates given by ${\mathtt n}^-_\theta(\zeta\in \dd u)$,  $ \theta\in \s$, see Proposition 2.5 in \cite{TAMS}, for further details. On the other hand, it also has a component which is absolutely continuous with respect to Lebesgue measure and instantaneous rates ${\mathtt a}^-_\theta$, $\theta\in \s$. When ${\mathtt a}^-_\theta>0$ for some $\theta$, the Lebesgue measure of the range of $\underline{L}^{-1}$ on any finite period of time is strictly positive with positive probability, and we say that $\underline{L}^{-1}$ has a `heavy range'. In that case, we have a.s.
\begin{equation}
\int_0^t \mathtt{a}_{\Theta_s}^-\dd\underline{L}_t = \int_0^t \one{\xi_s = \underline\xi_s}\dd s,\qquad  t\geq0.
\label{Revuz}
\end{equation}

\smallskip

Associated with the descending ladder processes, we can introduce the associated potential measure
\[
U^-_\theta (\dd r, \dd x, \dd \varphi) = \bE_{0,\theta}\left[\int_0^\infty \one{(\underline{L}^{-1}_{t}\in \dd r, \xi^{-}_{t}\in \dd x,\Theta^{-}_{t}\in \dd \varphi)}\dd t\right], \qquad r,x \geq0, \varphi\in \mathcal{S}.
\]
\begin{assumption}\label{ass4}The MAP ordinator $\xi$  is transient, satisfying a.s.
\[
\lim_{t\to\infty}\xi_t = \infty; \]
and  that the modulator 
\begin{equation}
 \label{prea2}
\Theta\text{ is positive recurrent with invariant distribution }\pi\text{ which is fully supported on }\s.
\end{equation}
\end{assumption}
A consequence of Assumption \ref{ass4} is that $-\underline\xi_\infty<\infty,$ $\mathbf{P}_{0,\theta}$ almost surely, for all $\theta\in\mathcal{S}$. We define
$
g_\infty = \sup\{t>0: \xi_t =\underline\xi_\infty\},
$
so that
\[
\underline{\xi}_\infty = \xi_{g_\infty}\text{ and }\underline{\Theta}_\infty = \Theta_{g_\infty},
\]
encode the MAP-analogue of the `point of  closest reach.' It is then a straightforward consequence of the last exit formula in \eqref{eq:last exit} that, for  $r,x \geq0$, $\varphi\in \mathcal{S}$,
\begin{align}
\mathbf{P}_{0,\theta} (g_\infty \in \dd r ,- \underline{\xi}_\infty \in \dd x, \,\underline{\Theta}_\infty\in \dd \varphi) &= 
\mathbf{E}_{0,\theta} \left[\sum_{g\in\underline{G}} \one{g\in \dd r,\, -\xi_{g-}\in \dd x, \, \Theta_{g-}\in \dd \varphi}\one{\zeta_g = \infty}\right]\notag\\
&=\mathbf{E}_{0,\theta} \left[\int_0^\infty\dd \underline{L}_s \one{s\in \dd r,\, -\xi_{s-}\in \dd x, \, \Theta_{s-}\in \dd \varphi}\right]{\mathtt n}^-_\varphi(\zeta = \infty)
\notag\\
&=U^-_\theta (\dd r, \dd x, \dd \varphi)
{\mathtt n}^-_\varphi(\zeta = \infty)\notag\\
&=\mathtt{U}^-_\theta (\dd r, \dd x, \dd \varphi),
\label{MAPPOCR}
\end{align}
where, for convenience, we have  defined 
\begin{equation}
\mathtt{U}^-_\theta (\dd r, \dd x, \dd \varphi) := U^-_\theta (\dd r, \dd x, \dd \varphi) 
{\mathtt n}^-_\varphi(\zeta = \infty), \qquad r,x \geq0, \varphi\in \mathcal{S}.
\label{MAPPOCR2}
\end{equation}

\smallskip

When it exists, let us also  define the invariant measure for $\Theta^-$ under $\hat{\mathbf{P}}$ by  $ \hat{\pi}^-$. {For conditions for this to happen, and expressions for $ \hat{\pi}^-$ we refer to \cite{TAMS}.}  {More generally, a hat on various quantities, e.g. $\hat{\mathtt{a}}^{-}_\theta$ will denote the same quantity without the hat albeit under $\hat{\mathbf{P}}$.}
We will need an additional  assumption that is related to Assumption \ref{ass4}. 
\begin{assumption}\label{ass5} The invariant measure $ \hat{\pi}^-$ exists and   
\[
c_{\hat{\pi}^{-}}:=\hat{\mathbf{E}}_{0,\pi}[\underline{L}^{-1}_1] =\hat{\mathtt{a}}^-_{\hat\pi^-}+ \hat{\mathtt n}^-_{\hat\pi^-}(\zeta) :=  \int_\mathcal{S} (\hat{\mathtt{a}}^-_\theta + \hat{\mathtt n}_{\theta}(\zeta))\hat\pi^-(\dd \theta)<\infty.
\]
 \end{assumption}
 
\noindent When $\xi$ is a Lévy process, it is known that the Assumption \ref{ass5} is equivalent to Assumption \ref{ass4}. Thus, heuristically speaking, Assumption \ref{ass5} should imply Assumption \ref{ass4} under certain necessary and sufficient regularity assumptions, which are yet to be completely determined in the literature.

{In \cite{TAMS} it has been proved that when \eqref{prea2} holds, and
\begin{equation}
\label{prea8}
\inf_{v\in\s}\left[\hat{\mathtt{a}}^{-}_v+\hat{\mathtt n}^-_{v}\left(1-\mathrm{e}^{-\zeta}\right)\right]>0\text{  and  }\hat{\mathtt n}^-_{v}(\zeta)<+\infty\text{ for every }v\in\s,
\end{equation}
then there is an invariant measure $\hat{\pi}^-$ and it is given by  
\begin{equation}\label{def:pi+}
\hat{\pi}^-(\cdot)=\frac{1}{c_{\hat{\pi}^{-}}}\hat{\bE}_{0,\pi}\left[\int_{0}^{1}  \one{\Theta_{s}\in\cdot}{\rm d}\underline{L}_{s}\right].
\end{equation}
 It  follows by \cite[Theorem (5.1)]{Kaspi1} that
\begin{equation}\label{kaspi-inv}
\pi({\rm d}v)=\frac{1}{c_{\hat{\pi}^{-}}}\left[\hat{\mathtt{a}}^{-}_v\hat{\pi}^{-}({\rm d}v)+\int_{\s }\hat{{\rm n}}^-_{\theta}\left(\int_{0}^{\zeta}\one{\Theta^\epsilon_{t}\in {\rm d}v}{\rm d}t\right)\hat{\pi}^{-}({\rm d}\theta)\right].
\end{equation}

}

\subsection{Absolute continuity assumptions}
We will need some assumptions asserting  absolute continuity of various measures associated to the MAP $(\xi, \Theta)$ with respect to ${\rm Leb}\otimes\pi$. The first assumption concerns the potential measures of $(\xi, \Theta).$

\begin{assumption}\label{ass6}The potentials of $((\xi, \Theta), \mathbf{P})$ and $((\xi, \Theta), \hat{\mathbf{P}})$ are absolutely continuous with respect to ${\rm Leb}\otimes\pi$, with 
$${\bE}_{x,\theta} \left[ \int_0^{\infty} \one{\xi_s \in \dd y, \Theta_s \in \dd\varphi} \dd t\right]=g((x,\theta), (y,\varphi))\pi(\dd \varphi)\dd y$$
$$\hat{\bE}_{x,\theta} \left[ \int_0^{\infty} \one{\xi_s \in \dd y, \Theta_s \in \dd\varphi} \dd t\right]=\hat g((x,\theta), (y,\varphi))\pi(\dd \varphi)\dd y.$$
 \end{assumption}
This assumption, together with the duality Assumption~\ref{ass2} imply that the densities $g$ and $\hat{g}$ can be chosen such that $$g((x,\theta), (y,\varphi))=\hat{g}((y,\varphi), (x,\theta)),\qquad x, y\in\R, \theta,\varphi\in \s. $$

\begin{remark} Notice that the assumption on the absolute continuity of the potential of $(\xi, \Theta)$ is satisfied as soon as the transition semigroup of $(\xi, \Theta)$ 
\[
\mathcal{P}_t ((x, \theta), (\dd z, \dd \beta)): = \mathbf{P}_{x,\theta}(\xi_t\in\dd z, \, \Theta_t\in \dd\beta), \qquad x,z\in \mathbb{R},\, \theta, \beta\in\mathcal{S}, t\geq0,
\] 
 is absolutely continuous with respect to ${\rm Leb}\otimes\pi.$ That is, 
 \[\mathcal{P}_t ((x, \theta), (\dd z, \dd \beta)) = p_t((x,\theta), (z,\beta)) \pi(\dd \beta) \dd z,
 \]
 for $t\geq0$, $x,z\in \mathbb{R}$ and $\theta, \beta\in\mathcal{S}$. Similiarly, the assumption on the absolute continuity of the potential of the dual process is satisfied as soon as its transition semigroup is absolutely continuous with respect to ${\rm Leb}\otimes\pi.$ 

\end{remark}

\smallskip

{In a similar spirit, we use an absolutely continuity assumption for the descending ladder potential.}
For both $U^-_\theta$, respectively $\mathtt{U}^-_\theta$, defined in \eqref{MAPPOCR},  resp. \eqref{MAPPOCR2},  with an abuse of notation, we will denote their space marginals by 
\begin{equation}\label{time-marginals}
\mathtt{U}^-_\theta ( \dd x, \dd \varphi):=\mathtt{U}^-_\theta ([0,\infty), \dd x, \dd \varphi), \quad \text{and}\quad U^-_\theta (\dd x, \dd \varphi):=U^-_\theta ([0,\infty), \dd x, \dd \varphi), \quad x\in \R, \varphi\in\mathcal{S}.
\end{equation}
To establish our main result, we will need to assume additionally that $\mathtt{U}_\theta$ (and hence $U_\theta^-$) has a density in the following two  different possible ways, the second being stronger than the first.

\begin{assumption}\label{ass7}
For all $x\geq0$ and $\theta, \varphi\in \mathcal{S}$, 
\[
\mathtt{U}^-_\theta ( \dd x, \dd \varphi) = {\mathtt u}^-_\theta(x, \dd \varphi)\dd x.
\]
\end{assumption}

\begin{assumption}\label{ass8}
There is a measure $\Xi$ on $\s$ such that for all $x\geq0$ and $ \theta, \varphi\in \mathcal{S}$, 
\[
\mathtt{U}^-_\theta ( \dd x, \dd \varphi) = {\mathtt u}^-_\theta(x, \varphi)\Xi(\dd \varphi)\dd x,
\] with ${\mathtt u}^-_\theta:[0,\infty)\times\s\mapsto[0,\infty),$ a continuous function for each $\theta\in \s.$
\end{assumption}

\section{Main results}

In order to state our main result, let us introduce further notation. Under Assumption \ref{ass7}, for a given measurable set $\Omega\subset\mathcal{S},$ define the function
\begin{equation}
\mathtt{h}^\downarrow(x, \theta; \Omega) := \int_\Omega {\mathtt u}^-_\theta(x, \dd \varphi), \qquad x\geq0, \theta\in \mathcal{S}.
\label{hdownarrow}
\end{equation}
Moreover, if the stronger Assumption \ref{ass8} is in force, we may further define 
\begin{equation}
\mathtt{h}^\downarrow(x, \theta; \varphi) := {\mathtt u}^-_\theta(x,  \varphi), \qquad x\geq0, \theta,\varphi \in \mathcal{S}.
\label{hdownarrow2}
\end{equation}
In both cases, we will denote
\begin{equation}
\mathtt{H}^\downarrow(x ; \Omega) := \mathtt{h}^\downarrow(\log\norm{x}, \arg(x); \Omega), \qquad x\in \mathcal{H}.
\label{Hdown}
\end{equation}
Finally, we define the $d$-dimensional sphere of radius $r$
\[
\mathbb{S}^{d-1}(r) = \{x\in\mathbb{R}^d: |x| = r\}, \qquad r>0.
\]

Leading up to the main result of this paper, Williams' type path decomposition of a self-similar Markov process, we first establish that the ssMp conditioned to continuously hit a patch on a sphere from one side can be defined via a limiting procedure. The ssMp conditioned to continuously hit the sphere will form the basis of the so-called `pre-minimum' process in the path decomposition. Moreover, the law of this process can be obtained from a Doob $h$-transform of the killed ssMp. 

\begin{theorem}\label{condition1}Suppose that Assumptions \ref{ass1}, \ref{ass3}, \ref{ass4} and \ref{ass7} hold. 
Fix a measurable set $\Omega\subset \mathcal{S}$ such that $\sigma_1(\Omega)>0,$ and $x\in\mathcal{H}$ and $y\in \R,$ such that $\norm{x}>{\rm e}^y$. Then we can define the ssMp conditioned to continuously hit the patch ${\rm e}^y\Omega: = \{z\in\mathbb{S}^{d-1}({\rm e}^y): \arg(z)\in \Omega\}$ without entering $\mathbb{S}^{d-1}({\rm e}^y)$ via the limit
\begin{equation}
\mathbb{P}_x^{\downarrow, (y,\Omega)}(A,\, t<\mathtt{k}) = \lim_{\varepsilon\to0}
\mathbb{P}_x(A, t < T_{\exp(y)} | \log\norm{X_{G(\infty)}} \in (y,y+\varepsilon), \arg(X_{G(\infty)})\in \Omega ),
\label{condtioningdown}
\end{equation}
which holds for any $A\in\mathcal{F}_t;$ where $\mathtt{k}$ denotes the path lifetime on $\mathbb{D}({\R\times\s})$,  
\[
G(\infty)= \sup\{t>0: \norm{X_t} = \inf_{s\geq0}\norm{X_s}\}\,\,\text{ and }\,\,
T_{z} = \inf\{t>0 : \norm{X_t}< z\}.
\]
We also have that 
\begin{equation}
\left.\frac{\dd \mathbb{P}^{\downarrow, (y,\Omega)}_x}{\dd \mathbb{P}_x}\right|_{\mathcal{F}_t\cap\{t<\zeta\}} 
= 
 \frac{\mathtt{H}^\downarrow(X_t{\rm e}^{-y}; \Omega)}{\mathtt{H}^\downarrow(x{\rm e}^{-y}; \Omega)}
 \one{t<T_{{\rm e}^{y}}}, \qquad t, \geq0, \norm{x}> {\rm e}^y.
\label{doobdown}
\end{equation}

If in addition, we assume Assumption \ref{ass8} (which supersedes Assumption \ref{ass7}), then \eqref{condtioningdown} and \eqref{doobdown} both make sense when $\Omega = \{\theta\}$, for $\theta\in\mathbb{S}^{d-1}$, providing we replace the conditioning in \eqref{condtioningdown} by the event $\{\norm{X_{G(\infty)} - e^y\theta}<\varepsilon \}$.
\end{theorem}
There are a number of remarks we can make concerning the assumptions and the statement of Theorem \ref{condition1}. Assumptions \ref{ass1} and \ref{ass3} is nothing more than a regularity assumption on the underlying MAP, to ensure that certain subtleties do not need to be accounted for. {An example of what the aforesaid regularity excludes would be the possibility of jumping away from local minima, which would result in excursions from the infimum starting with a jump in the additive part. The transience Assumption \ref{ass4} is necessary in order to have a global minimum. Assumptions \ref{ass7} and \ref{ass8} impose certain density requirements on the potentials ${\mathtt u}_\theta(\dd r, \dd \vartheta)$ which coincide with the ability to hit both radial positioned patches such as in the theorem, as well as individual points. The respective existence of the density of ${\mathtt U}_\theta(\dd r, \dd \vartheta)$ is not unusual when one considers that a similar need for the existence of a density of underlying potentials are required for similar questions in the theory of L\'evy processes, in particular in the setting of stable processes, and in general for building bridges of Markov processes. See, for example, Kesten \cite{kesten69} and Kyprianou et al. \cite{KPS}.

\smallskip

The  conditioning in Theorem \ref{condition1} features already in our Williams' path decomposition below. However, we need {still} another notion of conditioning, which is actually easier to construct, and which we describe below. 
Suppose Assumptions \ref{ass1} and  \ref{ass3},  \ref{ass4} hold. Then we can define the ssMp conditioned never to enter $\mathbb{S}^{d-1}({\rm e}^y)$ via
\[
\mathbb{P}_x^{\uparrow, (y)}(A) =\mathbb{P}_x(A  | T_{e^y} =\infty), \qquad A\in\mathcal{F}_t.
\]
In particular 
\begin{equation}
\left.\frac{\dd \mathbb{P}^{\uparrow, (y)}_x}{\dd \mathbb{P}_x}\right|_{\mathcal{F}_t\cap\{t<\zeta\}} 
= 
\frac{\mathtt{H}^\uparrow(X_t{\rm e}^{-y})}{\mathtt{H}^\uparrow(x{\rm e}^{-y})}\one{t<T_{e^y}}, \qquad t, \beta\geq0, \norm{x}>{\rm e}^y,
\label{doobup}
\end{equation}
where 
\begin{equation}
\mathtt{H}^\uparrow(x) = \mathbb{P}_x(T_{1}=\infty)>0, \qquad x\in \mathcal{H}.
\label{Hup}
 \end{equation}
 In addition, we can extend the definition of $\mathbb{P}_{x}^{\uparrow, (y)}, \norm{x}>{\rm e}^y,$ to allow for an entrance law on $\mathbb{S}^{d-1}({\rm e}^y)$ given by 
 \begin{equation}
 \mathbb{E}^{\uparrow, (y)}_x [f(X_t)]= \frac{{\mathtt n}^-_{\arg(x)}\left(
 f(
 {\rm e}^{\epsilon_{\varphi(t)}+y}\Theta^\epsilon_{\varphi(t)}
 )
 \mathtt{H}^\uparrow({\rm e}^{\epsilon_{\varphi(t)} }\Theta^\epsilon_{\varphi(t)}), \varphi(t)<\zeta  \right)
 }
 {
 {\mathtt n}^-_{\arg(x)} (\zeta = \infty)
 }, \qquad x\in \mathbb{S}^{d-1}({\rm e}^y),
 \label{doobup0}
 \end{equation}
 where   $f$ is a bounded measurable function  on $\{x\in\mathbb{R}^d: \norm{x}> {\rm e}^y\}$ and 
 \[
 \varphi(t) = \inf\left\{ s>0 : \int_0^s {\rm e}^{\alpha \epsilon_u}\dd u >t \right\}.
 \]

\begin{theorem}[Williams' path decomposition]\label{Wthrm} Suppose that Assumptions \ref{ass1}, \ref{ass2}, \ref{ass3}, \ref{ass4}, \ref{ass5}, \ref{ass6} and \ref{ass8} hold.  
Then $(X,\P)$ is equal in law to the following concatenation of processes:
for $x\in\mathcal{H},$ pick a point $X^*$ in $\mathcal{H},$ such that 
\begin{equation}
\mathbb{E}_x[f(X^*)] = \int_{[0,\infty)\times \mathcal{S}}f(\norm{x}{\rm e}^{-y}\varphi){\mathtt{U}^-_{\arg(x)}(\dd y, \dd \varphi)}. 
\label{x*}
\end{equation}
Given   $X^* = x^*$, construct the independent processes \[(X, \mathbb{P}_x^{\downarrow, (\log \norm{x^*} , \{\arg(x^*)\})})\quad\text{and}\quad (X, \mathbb{P}_{x^*}^{\uparrow, (\log\norm{x^*})}).\]
Then the concatenation of these two processes under the randomisation \eqref{x*} is equal in law to $(X, \mathbb{P})$.
\end{theorem}

The rest of this paper is organised as follows. In Section~\ref{section4} we establish and prove a key identity relating the potential of the excursion measure from the radial infimum of the dual MAP and the potential measure of the downward ladder height process, this is the content of Theorem~\ref{mehartheorem}. This complements our main results and we think it is of particular interest as it extends a very useful result of Silverstein~\cite{silverstein} for L\'evy processes to the MAP setting.  Then, in Subsection~\ref{4.2}, in the same spirit as in \cite{silverstein}, we prove in Theorem~\ref{thm:invariance} that the density of the potential measure of the downward ladder height process is an invariant function for the MAP killed at the first passage time below zero. Then we devote Subsection~\ref{4.2} to the proof of Theorem~\ref{condition1}. In Section~\ref{S5} we focus on proving Theorem~\ref{Wthrm}, by means of establishing the respective result for the underlying MAP. This is done making heavy use of the duality assumptions, jointly with the invariance result proved in Theorem \ref{thm:invariance}. We then focus on establishing the consequences of our results for two specific examples. They are the cases of Isotropic Stable processes in Section~\ref{Isostable} and  two-dimensional Brownian conditioned to stay in a cone, in Section~\ref{BMcone}. Although these are examples that have previously been considered in the literature, our general results for ssMps and MAPs  provide new information about these processes.     

\section{Calculations for MAPs}\label{section4}

In order to prove our main theorems in the previous section, we need to first address a number of related results for the MAP that underlies our ssMp. Ignoring the time change in the Lamperti--Kiu transformation, conditioning our ssMp to continuously hit 
a point (written in generalised polar coordinates) ${\rm e}^y\theta$, for $y\in\mathbb{R},$ and $\Omega\subset\mathcal{S}$, we see that there is a connection with $(\xi, \theta)$ continuously hitting the point $(y, \theta)$. Indeed, because of spatial homogeneity, without loss of generality, we can equivalently consider the case of $(\xi, \Theta)$ continuously hitting the point $(0, \theta)$.


\subsection{Excursions and duality}

We provide a new result as far as duality is concerned, which is analogous to a classical identity for L\'evy processes; see equation (3.3) of \cite{silverstein} or Exercise VI.6.5 p183 in Bertoin \cite{bertoin1998cambridge}. 


\begin{theorem}\label{mehartheorem} Suppose that the Assumptions  \ref{ass1},  \ref{ass2}, \ref{ass3}, \ref{ass4} and \ref{ass5} are satisfied. 
We have, for bounded measurable functions $f, h:\mathcal{S}\rightarrow \R$,  and $g:[0,\infty)\rightarrow \R$,
\begin{equation}
\begin{split}
    \label{occupation_identity}
   &\int_{\s}  \hat{\mu}^{-}(\dd\theta) h(\theta)\left\{\hat{\mathtt{a}}^-_\theta  f(\theta)g(0)+ \hat{\mathtt n}^-_{\theta}\left(\int^{\zeta}_{0}
    f(\Theta^\epsilon_s)g(\epsilon_{s}) \dd s \right)\right\}\\ 
    &\hspace{3cm}= \int_{\s}\pi(\dd\nu)f(\nu)\int_{[0,\infty)\times \s} \mathtt{U}_\nu^-(\dd y,\dd\theta)g(y)h( \theta),
    \end{split}
\end{equation}
where, 
\[
\hat{\mu}^-(\dd \theta) =\left( \hat{\mathtt{a}}^-_{\hat\pi^-}+ \hat{\mathtt n}_{\hat\pi^-}(\zeta)\right)^{-1}\hat{\pi}^-(\dd \theta),
\qquad
\theta\in \mathcal{S}.
\]
\end{theorem}

\begin{proof}
We use last exit decomposition and duality as a starting point of our calculation, and then we take limits. Indeed, for any real valued bounded measurable function $H$ defined on $\s \times [0,\infty) \times \s \times [0,\infty)^2 \times \s \times [0,\infty)$ we have
\begin{align*}
	\bE_{0,\pi}&\left[H(\Theta_t, \ubar{g}_t, \Theta_0, -\ubar{\xi}_t,t - \ubar{g}_t, \minmod{t}, \ordref{t}) \one{\ubar{g}_t<t} \right]
	\\
	&= \bE_{0,\pi} \left[ \int_0^t d\ubar{L}_s {\mathtt n}^-_{\Theta_s} \left( H(\Theta^\epsilon_{t-s},s, \Theta_0, \xi_s, t-s, \Theta_s, \epsilon_{t-s}) \one{t-s < \zeta}\right) \right]
	\\
	&= \int_{\s} \int_0^t \int_{[0,\infty) \times \s} \pi(\dd\nu){U}_\nu^-(\dd s, \dd y, \dd\theta) {\mathtt n}^-_{\theta} \left( H(\Theta^\epsilon_{t-s},s, \nu, \xi_s, t-s, \theta, \epsilon_{t-s}) \one{t-s < \zeta} \right).
\end{align*}
In particular, this means that for some non-negative and uniformly bounded functions   $G $ and $  K$ on $ [0,\infty)$ and $ \s,$ respectively, we have
\begin{align}
		&\bE_{0,\pi}\left[f(\Theta_{0})g(-\ubar{\xi}_t) h( \ubar{\Theta}_{t})G(\ordref{t})K(\Theta_t)\one{\ubar{g}_t<t}\right]
		\notag\\
		&=\int_{\s}\pi(\dd\nu)\int^{t}_0\int_{[0,\infty)\times \s}{U}^{-}_{\nu}(\dd s,\dd y,\dd \theta)f(\nu)g(y)h( \theta) {\mathtt n}^-_\theta\left(G\left(\epsilon_{t-s}\right)K(\Theta^{\epsilon}_{t-s})\one{t-s < \zeta}\right).
	\label{mainvariables}
\end{align}

Taking a version of \eqref{mainvariables} with fewer variables, renaming the functions  $K$ by $f$ and $G$ by $g$, and the role of $\mathbf{P}_{0,\pi}$ replaced by $\hat{\mathbf{P}}_{0,\pi}$, we have
\begin{align}
		&\hat{\bE}_{0,\pi}\left[f(\Theta_{t})g(\ordref{t}) h( \ubar{\Theta}_{t})\one{\ubar{g}_t<t}\right]
		\notag\\
		&=\int_{\s}\pi(\dd\nu)\int^{t}_0\int_{[0,\infty)\times \s}\hat{U}^{-}_{\nu}(\dd s,\dd y,\dd \theta) h(\theta)\hat{\mathtt n}^-_\theta\left(g\left(\epsilon_{t-s}\right)f(\Theta^{\epsilon}_{t-s})\one{t-s < \zeta}\right).
	\label{dual_mainvariables}
\end{align}

Next, we recall that, under Assumption \ref{ass4}, the dual process drifts to $-\infty$, upon which  an application of the Markov Renewal Theorem to the MAP given by $(\ubar{L}^{-1}, \Theta^-),$ cf, \cite{AlsmeyerTheMR}, reveals that
\begin{align}
		&\lim_{t\to\infty}\hat{\bE}_{0,\pi}\left[f(\Theta_{t})g(\ordref{t}) h( \ubar{\Theta}_{t})\one{\ubar{g}_t<t}\right] =
		\int_{\s}  \hat{\mu}^{-}(\dd\theta) h(\theta)\hat{\mathtt n}^-_{\theta}\left(\int^{\zeta}_{0}\dd s f(\Theta_{s}^\epsilon)g(\epsilon_{s}) \right).
		\label{MRTapplied}
\end{align}

\smallskip

To justify the use of the Markov Renewal Theorem it suffices to show that 
\begin{equation}
	\int_\s \sum_n \sup_{ \Delta n \leq t \leq \Delta ( n+1)} \hat{\mathtt n}^-_{\theta}\left( f(\Theta_{t}^\epsilon)g(\epsilon_{t})
	 \one{t<\zeta}\right)  \hat{\pi}^- (\dd \theta) < \infty,
\end{equation}
for some $\Delta>0$. 
To this end, we can bound $K$ and $G$ by unity without loss of generality, in which case, 
\begin{align}
	\int_\s \sum_n \sup_{ \Delta n \leq t \leq \Delta ( n+1)} \hat{\mathtt n}^-_{\theta}\left( f(\Theta_{t}^\epsilon)g(\epsilon_{t})
	{\one{t<\zeta}}\right)  \hat{\pi}^- (\dd \theta) 
	&\leq \int_\s    \sum_n  \hat{\mathtt n}_\theta^- \left(\Delta n <\zeta \right) \hat{\pi}^- (\dd \theta) 
	\notag\\
	&\leq \int_\s    \sum_n \frac{1}{\Delta} \int_{\Delta (n-1)}^{\Delta n} \hat{\mathtt n}_\theta^- \left(\ell <\zeta  \right)\dd \ell\,
	 \hat{\pi}^- (\dd \theta) 
	\notag\\
	&= \frac{1}{\Delta} \hat{\mathtt n}_{\hat{\pi}^-}^- \left(\zeta \right), 
	\label{MRTcondition}
\end{align}
which is finite by Assumption \ref{ass5}.

\smallskip
Next consider the asymptotic probability on the event that $\{t=\underline{g}_t\}$. We claim that the following holds true. 
\begin{align}
	&\lim_{t \to \infty}\hat{\bE}_{0,\pi}\left[f(\Theta_{t})g(\ordref{t}) h(\ubar{\Theta}_{t})\one{t=\underline{g}_t}\right]
	= 
	\int_\s f(\theta)g(0)h(\theta)\hat{\mathtt{a}}^-_\theta \hat{\mu}^-(\dd \theta).
	\label{integratetoprove}
\end{align}
Indeed, because $\pi$ is an invariant probability measure for $\Theta$, and by the limit in \eqref{MRTapplied}, we have
\begin{equation*}
\begin{split}
	&\hat{\bE}_{0,\pi}\left[f(\Theta_{t})g(\ordref{t}) h(\ubar{\Theta}_{t})\one{t=\underline{g}_t}\right]
	\\
	&=\hat{\bE}_{0,\pi}\left[f(\Theta_{t})g(0) h(\Theta_{t})\right]-\hat{\bE}_{0,\pi}\left[f(\Theta_{t})g(0) h(\Theta_{t})\one{\underline{g}_t<t}\right]\\
	&=g(0)\int_{\s}\pi(\dd\theta)f(\theta)h(\theta)-\hat{\bE}_{0,\pi}\left[f(\Theta_{t})g(0) h(\Theta_{t})\one{\underline{g}_t<t}\right]\\
&\xrightarrow[t\to\infty]{}g(0)\int_{\s}\pi(\dd\theta)f(\theta)h(\theta)-g(0)\int_{\s}  \hat{\mu}^{-}(\dd\theta) \hat{\mathtt n}^-_{\theta}\left(\int^{\zeta}_{0}\dd s f(\Theta_{s}^\epsilon)h(\Theta_{s}^\epsilon) \right)
\end{split}
\end{equation*}
To finish, we plug in the identity in \eqref{kaspi-inv} to conclude that 
\begin{equation*}
\begin{split}
&g(0)\left(\int_{\s}\pi(\dd\theta)f(\theta)h(\theta)-\int_{\s}  \hat{\mu}^{-}(\dd\theta) \hat{\mathtt n}^-_{\theta}\left(\int^{\zeta}_{0}\dd s f(\Theta_{s}^\epsilon)h(\Theta_{s}^\epsilon) \right)\right)\\
&=\frac{g(0)}{c_{\pi^{-}}}\int_{\s}\hat{\pi}^{-}({\rm d}\theta)\hat{\mathtt{a}}^{-}_{\theta}f(\theta)h(\theta)\\
&=g(0)\int_{\s}\hat{\mu}^{-}({\rm d}\theta)\hat{\mathtt{a}}^{-}_{\theta}f(\theta)h(\theta).
\end{split}
\end{equation*}
From the duality relationship of Lemma \ref{prop:equal dist}, we may then observe that 
\begin{align}
	&\lim_{t \to \infty}\hat{\bE}_{0,\pi}\left[f(\Theta_{t})g(\ordref{t}) h(\ubar{\Theta}_{t})\right]
	= \lim_{t \to \infty} \bE_{0,\pi} \left[ f(\Theta_{0})g(-\ubar{\xi}_t)h( \ubar{\Theta}_{t}) \right].
	\label{examine}
\end{align}
Accordingly, let us examine the limit on the right-hand side of \eqref{examine}.
{For, we go back to \eqref{mainvariables}, we set $G \equiv K \equiv 1$ there, and take limit as $t \to \infty$, recalling that, because of Assumption \ref{ass4}, ${U}^{-}_{\nu}(\dd y,\dd\theta)$ is a finite measure, and hence
\begin{align}
\lim_{t\to\infty}\bE_{0,\pi}\left[f(\Theta_{0})g(-\ubar{\xi}_t) h( \ubar{\Theta}_{t}) \right]
&=\int_{\s}\pi(\dd\nu)f(\nu)\int_{[0,\infty)\times \s}\bP_{0,\nu} ( -\ubar{\xi}_\infty\in \dd y, \, \ubar{\Theta}_\infty \in \dd \theta)g(y)h(\theta)
\notag\\
	&	=\int_{\s}\pi(\dd\nu)f(\nu)\int_{[0,\infty)\times \s}{U}^{-}_{\nu}(\dd y,\dd\theta)\mathtt{n}^-_\theta(\zeta=\infty)h(\theta)g(y)\notag\\
	&	=\int_{\s}\pi(\dd\nu)f(\nu)\int_{[0,\infty)\times \s}\mathtt{U}^{-}_{\nu}(\dd y,\dd\theta)h(\theta)g(y),
\label{otherdirection}
\end{align}
where we have used the definition in \eqref{MAPPOCR}. By comparing the two limits as given in  \eqref{MRTapplied} and \eqref{otherdirection}, we get the statement of the 
theorem. } 
\end{proof}

\begin{corollary} It is worth recording from the proof of part (i) of Theorem \ref{mehartheorem} that, for bounded measurable functions $f, h:\mathcal{S}\rightarrow \R$,  and $g:[0,\infty)\rightarrow \R$,
\begin{align*}
\bE_{0,\pi} \left[f(\Theta_0)g( -\ubar{\xi}_\infty)h(\ubar{\Theta}_\infty)\right]=\int_{\s}\pi(\dd\nu)f(\nu)\int_{[0,\infty)\times \s} \mathtt{U}_\nu^-(\dd y,\dd\theta)g(y)h( \theta)
    \end{align*}
\end{corollary}


\begin{remark}
\rm In the special case that our MAP is a L\'evy process, \eqref{occupation_identity} also reveals an identity which is not well stated in existing literature for L\'evy processes. Recall that the dual of $\xi$ in this setting is nothing more than $-\xi$. Hence, 
if $\mathcal{V}^-$ is the ascending ladder potential and $q^-$ is the rate at which excursions from the maximum are killed, moreover, if $n^+$ is the measure for excursions from the maximum, and $\delta^+$ is the drift of the inverse local time at the maximum, then we get
\begin{equation}
\frac{1}{\delta^++n^+(\zeta)} \left\{\delta^+  g(0)+ n^+\left(\int^{\zeta}_{0}
 g(\epsilon_{s}) \dd s \right)\right\} = q^- \int_{[0,\infty)}\mathcal{V}^-(\dd y)g(y).
 \label{Levytidy}
\end{equation}

Suppose we write $\Phi^+$ (resp. $\Phi^-$) for the Laplace exponent of the inverse local time at the maximum (resp. minimum).
It is well known from the Wiener--Hopf factorisation for L\'evy processes that $\lambda = \Phi^+(\lambda)\Phi^-(\lambda)$ for all $\lambda\geq0$. Since we are under the assumption that $\lim_{t\to\infty}\xi_t = \infty$, $\Phi^-(0) = q^- >0$ but $\Phi^+(0) = 0$. As a consequence, appealing to the Bernstein formula for Laplace exponents of subordinators,  
\[
1= \lim_{\lambda\to0}\frac{\Phi^+(\lambda)}{\lambda} \Phi^-(0) =( \delta^+ +n^+(\zeta) )q^-.
\]
As a consequence, \eqref{Levytidy} tidies up  to give us the well known identity (see e.g. \cite{silverstein} and Exercise VI.6.5 of \cite{bertoin1998cambridge})
\[
\delta^+  g(0)+ n^+\left(\int^{\zeta}_{0}
 g(\epsilon_{s}) \dd s \right) = \int_{[0,\infty)}\mathcal{V}^-(\dd y)g(y).
\]

It is normal in the L\'evy process literature to defined $g$ only as a function on $(0,\infty)$. In effect this is taking $g(0) = 0$, which reduces the identity to 
\[
 n^+\left(\int^{\zeta}_{0}
g(\epsilon_{s}) \dd s \right) = \int_{[0,\infty)}\mathcal{V}^-(\dd y)g(y).
\]

\end{remark}

\begin{remark}\rm
We can use a similar  argument to the classical one for L\'evy processes, given e.g. on pp158-159 of \cite{bertoin1998cambridge}, to  prove that the regularity Assumption \ref{ass3} ensures that $\hat{\mathtt{a}}^-_\theta = 0$ for all $\theta\in \s$ {but for a $\ubar{L}$-negligible set.} 
Indeed, under Assumption \ref{ass3}, using Lemma \ref{prop:equal dist}, we have that, for all $t>0$,
\begin{equation}
0={\bP}_{0,\pi}\left(\tau^-_0>t\right)={\bP}_{0,\pi}(\underline\xi_t=0)
=\hat{\bP}_{0,\pi}(\xi_t = \underline\xi_t).
\label{zeroallt}
\end{equation}
On the other hand, from \eqref{Revuz}, we have that
\begin{align*}
0=\int_0^\infty\hat{\bP}_{0,\pi}(\xi_t = \underline\xi_t)\dd t
&=\hat{\bE}_{0,\pi}\left[\int_0^\infty \hat{\mathtt{a}}_{\Theta_t}^-\dd \underline{L}_t \right]
=\int_{\s}\pi(\dd\theta)\int_{[0,\infty)\times \s}\hat{U}^-_{\theta} (\dd r, [0,\infty), \dd \varphi)\hat{\mathtt{a}}_{\varphi}^-,
\end{align*}
where the left-hand side uses \eqref{zeroallt}.
Since $\mathtt{a}_{\varphi}^-\geq 0$, for all $\varphi\in\s$, it follows that it is equal to zero a.e. with respect to the measure $$\int_{\s}\pi(d\theta)\int_{[0,\infty)\times \s}\hat{U}^-_{\theta} (\dd r, [0,\infty), \dd \varphi),$$ {which implies the claim.} 

\end{remark}

\subsection{Invariance  of $\mathtt{h}^\downarrow(x,\theta; \Omega)$}\label{4.2}
An important application of Theorem \ref{mehartheorem} is that the function $\mathtt{h}^\downarrow(r,\theta; \Omega)$ possesses an invariance property. Formally, it is harmonic, as per the following result. 
\begin{theorem} 
	\label{thm:invariance}
	Suppose that $\Omega\subset \mathcal{S}$ is Borel such  that $\sigma_1(\Omega)>0$. Under Assumptions \ref{ass1}, \ref{ass3}, \ref{ass4}, \ref{ass5}, and \ref{ass7} , we have that 
\begin{equation}
\label{omegaortheta}
    \mathtt{h}^\downarrow( y,\varphi; \Omega) = \bE_{y, \varphi} \left[\mathtt{h}^\downarrow(\xi_{\tau_A}, \Theta_{\tau_A}; \Omega) \one{{\tau_A} < \tau_0^-} \right], \qquad  {\rm Leb}\otimes\pi   \text{-a.e.},
\end{equation}
where $\tau^-_0 = \inf\{t>0: \xi_t<0\}$ and $\tau_A = \inf\{t>0 : \xi_t \in A\}$ for any Borel set $A$. 
{Moreover, if we further impose Assumption \ref{ass8} in place of Assumption \ref{ass7}, then  we can replace $\Omega$ by a singleton $\{\theta\}$ in \eqref{omegaortheta}, {with the equality holding $ {\rm Leb}\otimes\pi  \otimes \pi$-almost everywhere}.}
\end{theorem}
\begin{proof}

We first define the region $$\mathrm{D}_n : = \{ (x,\theta) \in \R^d \times \Omega : 1/n \leq |x| \leq n \}$$ 
and for each $n$ we construct a measure $\ell^\Omega_n$ defined via the identity
\begin{equation}
	\label{ell_measure}
	\int_{[0,\infty) \times \Omega} \ell^\Omega_n (\dd x,\dd\theta) f(x,\theta) = \hat{\mathtt n}^-_{\hat\mu^-_\Omega}\left( \one{T_{\mathrm{D}_n} < \zeta} f(\epsilon_{T_{\mathrm{D}_n}} , \Theta^\epsilon_{T_{\mathrm{D}_n}}) \right),
\end{equation}
{where $f$ is non-negative and uniformly bounded from above, $\hat\mu^-_\Omega$ denotes the restriction of $ \hat{\mu}^-$ to $\Omega,$ }
$$\hat{\mathtt n}^-_{\hat\mu^-_\Omega}\left(\cdot\right)=\int_{\s}\hat\mu^-(d\varphi)\one{\varphi\in \Omega}\hat{\mathtt n}^-_{\varphi}\left(\cdot\right)$$
and $T_{{\rm D}_n}$ denotes the hitting time of the region ${\rm D}_n.$ {Next, we define the potential or $0$-resolvent for the MAP $(\xi, \Theta)$ under $\hat \bP$,  killed when $\xi$ enters the lower half-line, as}
\begin{equation*}
	\hat{G}^\dagger(x,\theta, \dd y, \dd\varphi) = {\hat{\bE}}_{x,\theta} \left[ \int_0^{\tau^-_0} \one{\xi_s \in \dd y, \Theta_s \in \dd\varphi} \dd t\right].
\end{equation*}
For the aforesaid  killed process, observe that {by the Markov property under $\hat{\mathtt n}^-_{\cdot}$ }
\begin{equation}
	\int_{[0,\infty) \times \Omega \times [0,\infty) \times {\color{amethyst}\mathcal{S}}} \ell^\Omega_n(\dd x, \dd\theta) \hat{G}^\dagger(x,\theta, \dd y, \dd\varphi) f(y,\varphi) = \hat{\mathtt n}^-_{\hat\mu^-_\Omega} \left( \int_{T_{\mathrm{D}_n}}^\zeta \dd s f(\epsilon_{s} , \Theta^\epsilon_{s}) \right).
	\label{boundednessneeds}
\end{equation}
Under Assumption \ref{ass7}, recalling the definition of $\mathtt{h}^\downarrow$ given in \eqref{hdownarrow}, we can now rewrite (\ref{occupation_identity}) as 
\begin{equation}
	\label{occupation_with_h}
	\hat{\mathtt n}^-_{\hat\mu^-_\Omega} \left( \int_{0}^\zeta \dd s f(\epsilon_{{s}} , \Theta^\epsilon_{s}) \right) = \int_\s \pi(\dd\nu)\int_0^\infty \dd y \mathtt{h}^\downarrow( y,\nu; \Omega) f(y, \nu).
\end{equation}
{Notice that by construction, as $n\to\infty,$ $T_{D_n}\rightarrow 0$ a.s.  By the monotone convergence theorem, we can assemble the pieces together, to get} 
\begin{align}
	\lim_{n \to \infty} &\int_{[0,\infty) \times \Omega \times [0,\infty) \times {\color{amethyst}\mathcal{S}}} \ell^\Omega_n(\dd x, \dd\theta) \hat{G}^\dagger(x,\theta, \dd y, \dd\varphi)f(y,\varphi)
	\notag\\
	&= \int_{[0, \infty) \times \s}  \dd y \pi(\dd\varphi) \mathtt{h}^\downarrow(y,\varphi; \Omega) f(y,\varphi),
	\label{getdensity}
\end{align}
where the limit occurs in an increasing fashion. {Next we claim that Assumption \ref{ass6} ensures that there exists a 
$\hat{g}^\dagger:\mathbb{R}\times\mathcal{S}\times\mathbb{R}\times\mathcal{S}\to \mathbb{R}^{+}$ measurable, such that
\begin{equation*}
	\hat{G}^\dagger(x,\theta, \dd y, \dd\varphi) = \hat{g}^\dagger( (x,\theta), (y, \varphi)) \dd y \pi(\dd\varphi).
\end{equation*}
Indeed, decomposing the resolvent $\hat G$ into its pre and post killing parts, we find
\begin{align*}
	\hat{G}(x,\theta,\dd z, \dd\beta) = \hat{G}^\dagger(x,\theta, \dd z, \dd\beta)  + \int_{[0,\infty) \times \s} \bP_{x,\theta}((\xi_{\tau^-_0  }, \Theta_{\tau^-_0 }) \in (\dd y, \dd\varphi)) \hat G(y,\varphi, \dd z, \dd\beta),
\end{align*}
where $\tau^-_0  = \inf\{t>0: \xi_t<0\}$. In turn, this gives us a domination of the measures $$\hat{G}(x,\theta,\dd z, \dd\beta)\geq \hat{G}^\dagger(x,\theta, \dd z, \dd\beta),$$ which is enough to ensure the absolute continuity of the measure $\hat{G}^\dagger.$ }

\smallskip

Returning to \eqref{getdensity}, we now have that the following limit holds {in the monotone increasing sense:}
\begin{equation}
	\label{limit_g_dagger}
	\lim_{n \to \infty} \int_{[0,\infty) \times \s} \ell^\Omega_n (\dd x, \dd\theta) \hat{g}^\dagger( (x,\theta), (y, \varphi)) =  \mathtt{h}^\downarrow(y,\varphi; \Omega), \qquad {\rm Leb}\otimes\pi  \text{-a.e.}
\end{equation}
In the same vein as \cite{silverstein}, we appeal to Theorem VI.1.16 of Blumenthal and Getoor \cite{BlumenthalGetoor}  to  get
\begin{equation}
\begin{split}
	\label{BGduality}
	&\int_{[0,\infty) \times \s}  \hat{\mathcal{P}}_A ((x, \theta), (\dd z, \dd\beta)) \hat{g}^\dagger((z,\beta), (y,\varphi)) \\
	&= \int_{[0,\infty) \times \s}  {\mathcal{P}}_A ((y,\varphi),(\dd z,\dd\beta)) \hat{g}^\dagger ( (x,\theta),(z,\beta)),
\end{split}	
\end{equation}
where $${\mathcal{P}}_A ((y,\varphi),(\dd z,\dd\beta))=\mathbf{P}_{y,\varphi}(\xi_{\tau_A}\in \dd z, \Theta_{\tau_A}\in \dd\beta),$$  and $\tau_A := \inf \{ t \geq 0 : (\xi_t , \Theta_t) \in A\}.$ Similarly, with $\hat{{\mathcal{P}}}_A$ defined under $\hat{\mathbf{P}}.$ 

Finally, with the aim of showing the harmonic property of $\mathtt{h}^\downarrow$, for some $A \subseteq [0,\infty) \times \s,$ define the hitting time $\tau_A := \inf \{ t \geq 0 : (\xi_t , \Theta_t) \in A\}.$ We compute for non-negative and uniformly bounded $f$,  
\begin{align*}
	&\int_{[0,\infty) \times \s}  \dd y \pi(\dd \varphi) f(y, \varphi) \bE_{y, \varphi} \left[ \mathtt{h}^\downarrow(\xi_{\tau_A}, \Theta_{\tau_A};\Omega) \one{\tau_A < \tau_0^-} \right]
	\\
	&= \int \dd y \pi(\dd \varphi) f(y,\varphi) \bE_{y, \varphi} \left[ \lim_{n\to \infty} \int_{[0,\infty) \times \s}  \ell^\Omega_n(\dd x,\dd \theta) \hat{g}^\dagger( (x,\theta), (\xi_{\tau_A}, \Theta_{\tau_A})) \one{\tau_A < \tau_0^-} \right]
	\\
	&= \lim_{n\to \infty}  \int_{[0,\infty) \times \s}  \ell^\Omega_n(\dd x,\dd \theta) \int_{[0,\infty) \times \s}   \dd y \pi(\dd \varphi) f(y,\varphi) \bE_{y,\varphi} \left[ \hat{g}^\dagger( (x,\theta), (\xi_{\tau_A}, \Theta_{\tau_A})) \one{\tau_A < \tau_0^-}  \right] 
	\\
	&= \lim_{n\to \infty} \int_{[0,\infty) \times \s}  \ell^\Omega_n(\dd x,\dd \theta) \int_{[0,\infty) \times \s}  \dd y \pi(\dd \varphi) f(y, \varphi) \int {\mathcal{P}}_A((y,\varphi), (\dd z, \dd \beta)) \hat{g}^\dagger( (x,\theta),(z,\beta))
	\\
	&= \lim_{n \to \infty} \int_{[0,\infty) \times \s}  \ell^\Omega_n(\dd x,\dd \theta)\int_{[0,\infty) \times \s}  \dd y \pi(\dd \varphi) f(y, \varphi) \int \hat{\mathcal{P}}_A((x,\theta), (\dd z, \dd \beta) ) \hat{g}^\dagger(  (z, \beta), (y,\varphi))
	\\
	&= \lim_{n\to\infty} \int_{[0,\infty) \times \s}  \ell^\Omega_n(\dd x,\dd \theta)\int_{[0,\infty) \times \s}  \dd y \pi(\dd \varphi) f(y, \varphi) \dE_{x,\theta} \left[ \hat{g}^\dagger( (\xi_{\tau_A}, \Theta_{\tau_A}), (y,\varphi) ) \one{{\tau_A} < \tau_0^-} \right]
	\\
	&= \lim_{n\to\infty} \int_{[0,\infty) \times \s}  \ell^\Omega_n(\dd x,\dd \theta) \dE_{x,\theta} \left[ \hat{G}^\dagger[f](  \xi_{\tau_A}, \Theta_{\tau_A} ) \one{{\tau_A} < \tau_0^-} \right]
	\\
	&=  \hat{\mathtt n}^-_{\hat\mu^-_\Omega} \left( \int_{0}^\zeta  f(\epsilon_{t} , \Theta^\epsilon_{t}) \dd t \right)
	\\
	&= \int_{[0,\infty) \times \s} \dd y \pi(\dd \varphi)f(y, \varphi) \mathtt{h}^\downarrow( y,\varphi; \Omega),
\end{align*}
{where we have made use of the monotone limit (\ref{limit_g_dagger}) to exchange limits and integration}, the duality relationship (\ref{BGduality}), and the occupation identity for the excursion measure (\ref{occupation_with_h}).   

\smallskip

We now have that for any bounded and measurable function $f$, 
\begin{equation}
	\mathtt{h}^\downarrow (y,\varphi; \Omega) = \bE_{y, \varphi} \left[ \mathtt{h}^\downarrow(\xi_{\tau_A}, \Theta_{\tau_A}; \Omega) \one{{\tau_A} < \tau_0^-} \right], \qquad {\rm Leb}\otimes\pi  \text{-a.e.}
	\label{harmonicae}
\end{equation}
i.e. $\mathtt{h}^\downarrow $ is a harmonic function. For the case where $\Omega$ is the singleton $\{ \theta \}$, it is not enough to use Assumption \ref{ass7},  so we use instead Assumption \ref{ass8}. As a slight abuse of notation, for some bounded measurable function $\alpha,$ we write 
\begin{equation*}
	\mathtt{h}^\downarrow (y,\varphi; \alpha) = \int_\Omega \mathtt{h}^\downarrow (y,\varphi; \{\theta \}) \alpha(\theta) \pi(\dd \theta).
\end{equation*}

Following the calculations between \eqref{ell_measure} and \eqref{harmonicae}, albeit replacing $\mathtt{h}^\downarrow (y,\varphi; \Omega) $ by  $\mathtt{h}^\downarrow (y,\varphi; \alpha) $, $\hat\mu^-_\Omega(d\theta)$ by $\alpha\hat\mu^-(d\theta):=\alpha(\theta)\hat\mu^-_\Omega(d\theta),$ and  $\ell^\Omega_n(\dd x, \dd\theta) $ by 
\[
\ell^\alpha_n(\dd x, \dd\theta): = \hat{n}^-_{\alpha\hat\mu^-}(\epsilon_{T_{D_n}}\in \dd x, \Theta^\epsilon_{T_{D_n}} \in \dd \theta),
\]
 we 
obtain the analogue of \eqref{harmonicae}, that is 
\[
\mathtt{h}^\downarrow (y,\varphi; \alpha) = \bE_{y, \varphi} \left[ \mathtt{h}^\downarrow(\xi_{\tau_A}, \Theta_{\tau_A}; \alpha) \one{{\tau_A} < \tau_0^-} \right], \qquad {\rm Leb}\otimes\pi  \text{-a.e.}
\]
Since $\alpha$ is any bounded measurable function, it follows that 
\begin{equation*}
	\mathtt{h}^\downarrow (y,\varphi; \{\theta \}) = \bE_{y,\varphi} \left[ 	\mathtt{h}^\downarrow (\xi_{\tau_A}, \Theta_{\tau_A}; \{\theta\})\one{\tau_A < \tau_0^-} \right], \qquad  {\rm Leb}\otimes\pi \otimes\pi   \text{-a.e.,}
\end{equation*}
which concludes the proof.
\end{proof}

\subsection{Proof of Theorem \ref{condition1}}
 In this section, we prove Theorem \ref{condition1} by first showing a  pre-cursor  result. That is to say we prove the analogous result to Theorem \ref{condition1} for the underlying MAP $(\xi,\Theta)$. Thereafter we justify why this is a sufficient result to give the statement of Theorem \ref{condition1}. 
 
 \smallskip

Our objective is thus to  characterise the law  of the MAP $(\xi, \Theta)$ conditioned to continuously hit $\{0\}\times \Omega$, without $\xi$ entering $(-\infty, 0)$, for some Borel $\Omega\subset \mathcal{S}$  satisfying $\sigma_1(\Omega)>0$, where $\sigma_1$ denotes the Lebesgue surface measure on $\mathbb{S}^{d-1}$, or for $\Omega = \{\theta\}$, for some fixed $\theta\in \mathcal{S}$. 

\smallskip

For the case that $\Omega$ satisfies $\sigma_1(\Omega)>0$.  We describe the law of this process via the limiting procedure 
\begin{equation*}
    \bP_{x,\theta}^{\Omega} (A, t < \zeta)  =  \bP_{x,\theta}^{\Omega}(A, t < \tau^{-}_0) :  =\lim_{\varepsilon \to 0} \bP_{x,\theta}(A, t < \tau^{-}_0 | \ubar{\xi}_{\infty} \in [0,\varepsilon), \ubar{\Theta}_\infty \in \Omega),
\end{equation*}
where $A \in \mathcal{G}_t$, where $(\mathcal{G}_t, t\geq 0)$ is the  is the minimal augmented admissible filtration  induced by the MAP $(\xi, \Theta)$ and 
\[
\tau^{-}_{0}= \inf\{t>0: \xi_t <\beta\}.
\]
In the setting that $\Omega = \{\theta\}$, we can adjust the above limit to 
\begin{equation*}
    \bP_{x,\theta}^{\Omega} (A, t < \zeta)  =  \bP_{x,\theta}^{\Omega}(A, t < \tau^{-}_0) :  = \lim_{\varepsilon \to 0} \bP_{x,\theta}(A, t < \tau^{-}_0 | \ubar{\xi}_{\infty} \in (0,\varepsilon), \norm{\ubar{\Theta}_\infty - \theta}<\varepsilon),
\end{equation*}
Under Assumptions \ref{ass1}, \ref{ass3}, \ref{ass4} and \ref{ass7},
we also have
 \begin{equation}
  \left.   \frac{\dd \bP_{x,\theta}^{\Omega}}{\dd \bP_{x,\theta}} \right|_{\mathcal{G}_t\cap\{t<\zeta\}}= \frac{\mathtt{h}^\downarrow(\xi_{t}, \Theta_{t}; \Omega)}{\mathtt{h}^\downarrow(x, \theta; \Omega)}\one{t<\tau^{-}_0}, \qquad t>0, x>0, \theta\in \mathcal{S}.
  \label{lawdownarrow}
 \end{equation} 
 Moreover, if we replace Assumption \ref{ass8} in place of Assumption \ref{ass7}, we have the same conclusion with $\Omega = \{\theta\}$ for $\theta\in\mathcal{S}$.

\smallskip

We first prove the setting that $\Omega$ is such that $\sigma_1(\Omega)>0$. The proof for the second setting is similar. 
	Making use of the Markov property and Fatou's lemma, we see that
	\begin{align*}
		\liminf_{\varepsilon \to 0}  &\bP_{x,\theta}(A, t < \tau^{-}_{0} | \xi_{g_\infty} \in (0,\varepsilon), \Theta_{g_\infty} \in \Omega)
		\\
		&= \liminf_{\varepsilon \to 0} \bE_{x,\theta} \left[ \one{A, t< \tau^{-}_{0} } \frac{\bP_{\xi_t, \Theta_t}(\xi_{g_\infty} \in (0,\varepsilon), \Theta_{g_\infty} \in \Omega)}{\bP_{x,\theta}(\xi_{g_\infty} \in (0,\varepsilon), \Theta_{g_\infty} \in \Omega)} \right]
		\\
		&\geq \bE_{x,\theta} \left[ \one{A, t < \tau^{-}_{0} } \liminf_{\varepsilon \to 0} \frac{\bP_{\xi_t, \Theta_t}(\xi_{g_\infty} \in (0,\varepsilon), \Theta_{g_\infty} \in \Omega)}{\bP_{x,\theta}(\xi_{g_\infty} \in (0,\varepsilon), \Theta_{g_\infty} \in \Omega)}\right]
		\\
		&= \frac{1}{\mathtt{h}^\downarrow(x, \theta ;, \Omega)}\textbf{E}_{x,\theta} \left[ \mathtt{h}^\downarrow(\xi_{t}, \Theta_{t}; \Omega) \one{A, t < \tau^{-}_{0}} \right],
	\end{align*}	
where we have used the fact that, 
\begin{align*}
	&\lim_{\varepsilon \to 0} \frac{\bP_{y,\varphi} ( \xi_{g_\infty} \in (0,\varepsilon), \Theta_{g_\infty} \in \Omega  )}{\bP_{x,\theta} ( \xi_{g_\infty} \in (0,\varepsilon), \Theta_{g_\infty} \in \Omega  )} 
	\\
	&= \lim_{\varepsilon \to 0} \frac{\mathtt{U}^-_{\varphi} (y- (0, \varepsilon), \Omega)}{\mathtt{U}^-_{\theta} ( x-(0, \varepsilon), \Omega)}
	\\
	&= \lim_{\varepsilon \to 0} \frac{\int_{\{ y-(0, \varepsilon) \}\times \Omega} \mathtt{u}_{\varphi}^-(z ,\dd\theta) dz  }{\int_{\{ x-(0, \varepsilon)\} \times \Omega} \mathtt{u}_{\theta}^-(z ,\dd\theta) \dd z  }
	= \lim_{\varepsilon \to 0} \frac{\frac{1}{\varepsilon}\int_{y-\varepsilon}^{y} \int_\Omega \mathtt{u}^-_{\varphi}(z ,\dd\theta) \dd z  }{\frac{1}{\varepsilon}\int_{x-\varepsilon}^{x} \int_\Omega \mathtt{u}^-_{\theta}(z ,\dd\theta) \dd z  }
	= \frac{\int_\Omega \mathtt{u}^-_{\varphi} (y, \dd\theta)  }{\int_\Omega \mathtt{u}^-_{\theta} (x, \dd\theta)  } =\frac{\mathtt{h}^\downarrow(y, \varphi; \Omega)}{\mathtt{h}^\downarrow(x, \theta; \Omega)},
\end{align*}
which is a consequence of Lebesgue's differentiation theorem.
If we change $A$ to $A^c,$ we obtain 
	\begin{align*}
		&\frac{1}{\mathtt{h}^\downarrow(x , \theta ; \Omega)}\textbf{E}_{x,\theta} \left[ \mathtt{h}^\downarrow(\xi_{t}, \Theta_t ; \Omega) \one{A^c, t < \tau^{-}_{0}} \right]
		\\
		&\leq \liminf_{\varepsilon \to 0} \bP_{x,\theta}(A^c, t < \tau^{-}_{0} | \xi_{g_\infty} \in (0,\varepsilon), \Theta_{g_\infty} \in \Omega)
		\\
		&= \liminf_{\varepsilon \to 0} \bP_{x,\theta}(t < \tau^{-}_{0} | \xi_{g_\infty} \in (0,\varepsilon), \Theta_{g_\infty} \in \Omega) \\&\quad- \limsup_{\varepsilon \to 0} \bP_{x,\theta}(A, t < \tau^{-}_{0}| \xi_{g_\infty} \in (0,\varepsilon), \Theta_{g_\infty} \in \Omega)
		\\
		&\leq 1 - \limsup_{\varepsilon \to 0} \bP_{x,\theta}(A, t <\tau^{-}_{0} | \xi_{g_\infty} \in (0,\varepsilon), \Theta_{g_\infty} \in \Omega)
		\\
		&=  \frac{1}{\mathtt{h}^\downarrow(x, \theta ; \Omega)}\textbf{E}_{x,\theta} \left[\mathtt{h}^\downarrow(\xi_{t}, \Theta_{t}; \Omega) \one{ t < \tau^{-}_{0}} \right]  - \limsup_{\varepsilon \to 0} \bP_{x,\theta}(A, t < \tau^{-}_{0} | \xi_{g_\infty} \in (0,\varepsilon), \Theta_{g_\infty} \in \Omega),
	\end{align*}
	which implies that 
		\begin{align*}
		&\limsup_{\varepsilon \to 0} \bP_{x,\theta}(A, t < \tau^{-}_0 | \xi_{g_\infty} \in (0,\varepsilon), \Theta_{g_\infty} \in \Omega)
		\\
		&\leq \frac{1}{\mathtt{h}^\downarrow(x, \theta ; \Omega)} \left(\textbf{E}_{x,\theta} \left[\mathtt{h}^\downarrow(\xi_{t}, \Theta_{t}; \Omega) \one{t < \tau^{-}_0} \right] - \textbf{E}_{x,\theta} \left[ \mathtt{h}^\downarrow(\xi_{t}, \Theta_{t}; \Omega) \one{A^c, t < \tau^{-}_0} \right] \right)
		\\
		&= \frac{1}{\mathtt{h}^\downarrow(x, \theta ; \Omega)}\textbf{E}_{x,\theta} \left[ \mathtt{h}^\downarrow(\xi_{t}, \Theta_{t}; \Omega) \one{A, t < \tau^{-}_0} \right] ,
	\end{align*}
	where we have used the fact that the function $\mathtt{h}^\downarrow$ is invariant from Theorem \ref{thm:invariance}. Combining these inequalities, monotonicity reveals that 
	\begin{equation*}
		\lim_{\varepsilon \to 0}\bP_{x,\theta}(A, t < \zeta | \xi_{g_\infty} \in (0,\varepsilon), \Theta_{g_\infty} \in \Omega) =  \frac{1}{\mathtt{h}^\downarrow(x, \theta ; \Omega)}\textbf{E}_{x,\theta} \left[ \mathtt{h}^\downarrow(\xi_{t}, \Theta_{t}; \Omega) \one{A, t < \tau_0^-} \right],
	\end{equation*}
giving us the change of measure as in the statement of the theorem. 

\smallskip

Finally, we can  now show that the ssMp conditioned to hit the patch $\Omega$ will indeed inherit the change of measure established for the corresponding MAP. As above, we give the proof for $\Omega$ such that $\sigma_1(\Omega)>0$ and leave the proof for the case in which $\Omega$ is a singleton as a straightforward adaptation of the proofs. 

\smallskip

We recall that the change of measure between $\bP^{\Omega}$ and $\bP$ is given by the Doob $h$-transform
 \begin{equation*}
		  \left.   \frac{\dd \bP_{x,\theta}^{\Omega}}{\dd \bP_{x,\theta}} \right|_{\mathcal{G}_t\cap\{t<\zeta\}} = \frac{\mathtt{h}^\downarrow(\xi_{t}, \Theta_{t}; \Omega)}{\mathtt{h}^\downarrow(x, \theta; \Omega)}\one{t < \tau^{-}_0}, 
		  \qquad t,x>0, \theta\in \s.
\end{equation*}
It can be easily verified that for any $t\geq 0,$ the time change $\varphi(t),$ is a stopping time. Theorem 3.4 from \cite{jacod2013limit} then allows us to transfer the martingale property of the Doob $h$-transform for the MAP to the Radon-Nikodym density for the ssMp. Indeed,
\begin{align*}
	 \left. \frac{\dd \P_{x}^{\Omega}}{\dd \P_{x}} \right|_{\mathcal{G}_{t}\cap\{t<\zeta\}}= \left.   \frac{\dd \bP_{\log\norm{x},\arg{x}}^{\Omega}}{\dd \bP_{\log\norm{x},\arg{x}}} \right|_{\mathcal{G}_{\varphi(t)}\cap\{\varphi(t)<\zeta\}} &= \frac{\mathtt{h}^\downarrow(\xi_{\varphi(t)}, \Theta_{\varphi(t)}; \Omega)}{\mathtt{h}^\downarrow(\log\norm{x},\arg{x}; \Omega)}
	  \\
	  &= \frac{\mathtt{h}^\downarrow(\log \norm{X_t}, \arg (X_t); \Omega)}{\mathtt{h}^\downarrow(\log \norm{x}, \arg (x); \Omega)}
	  \\
	  &= \frac{\mathtt{H}^\downarrow(X_t, \Omega)}{\mathtt{H}^\downarrow(x,\Omega)}.
\end{align*}
To complete the 
proof it suffices to note that the more general $\mathbb{P}_x^{\downarrow, (y,\Omega)}$ can be constructed from  $\mathbb{P}_x^{\Omega}$  by scaling, and hence that \eqref{doobdown} holds. \hfill$\square$

\section{Proof of Theorem \ref{Wthrm}}\label{S5}
Similarly to the proof of Theorem \ref{condition1}, because of the Lamperti--Kiu transform \eqref{eq:lamperti_kiu}, it suffices to prove the analogue of Theorem \ref{Wthrm} for the underlying MAP $(\xi, \Theta)$. More precisely,  this means that a pair $(\underline\xi^*, \underline\Theta^*)$ is selected such that $(-\underline\xi^* +y, \underline\Theta^*)$ have probability laws $\mathtt{U}^{-}(x, \varphi),$ $x\geq0$, $\varphi\in \mathcal{S}$. Given the pair $(\underline\xi^*, \underline\Theta^*) = (x^*, \theta^*)$, for $\nu\in\mathcal{S}$, construct two independent processes under $\mathbf{P}^{\downarrow, \theta^*}_{y-x^*, \nu}$ and $\mathbf{P}_{x^*, \theta^*}^{\uparrow}$. 
Then the concatenation of $(\xi+x^*, \Theta)$ under $\mathbf{P}^{\downarrow, \theta^*}_{y-x^*, \theta}$ and $(\xi, \Theta)$ under $\mathbf{P}_{x^*, \theta^*}^{\uparrow}$ under the aforesaid randomisation of the pair $(x^*, \theta^*)$ gives a process equal in law to $(\xi, \Theta)$ under  $\mathbf{P}_{y,\nu}$.  Here, we understand  $\mathbf{P}^{\downarrow, \theta^*}_{y-x^*, \nu}$ to mean the law of the MAP conditioned to continuously approach its minimal value and $\mathbf{P}_{x^*, \theta^*}^{\uparrow}$ to mean the law of the MAP conditioned to stay above $x^*$, when issued from $(x^*, \theta^*)$. Their precise meaning and consistency with  $\mathbb{P}_x^{\downarrow, (\log \norm{x^*} , \{\arg(x^*)\})}$
and $ \mathbb{P}_{x^*}^{\uparrow, (\log\norm{x^*})}$ will be elaborated on below.

\smallskip


Recall  from \eqref{lawdownarrow} that the process conditioned to continuously hit the patch $\Omega: = \{z : \norm{z}=y, \arg(z)\in \Omega\}$ without entering it is defined via the following change of measure with respect to the law of the killed process:
\begin{equation*}
\begin{split}
	&{\bE}_{x,\nu}^{\downarrow, (y, \Omega)}\left(F\left(\xi_s, \Theta_s,  s\leq t\right)\one{ t < \zeta }\right) \\
	&= \frac{1}{\mathtt{h}^\downarrow(x-y, \nu;  \Omega)} 	{\bE}_{x,\nu}\left(F\left(\xi_s, \Theta_s,  s\leq t\right)\mathtt{h}^\downarrow(\xi_t-y, \Theta_t;  \Omega)\one{ t <  \tau_y^- }\right)\\
	&= \frac{1}{\mathtt{h}^\downarrow(x-y, \nu;  \Omega)} 	{\bE}_{x-y,\nu}\left(F\left(\xi_s+y, \Theta_s,  s\leq t\right)\mathtt{h}^\downarrow(\xi_t, \Theta_t;  \Omega)\one{t <  \tau_0^- }\right), \quad x>y, t > 0, \nu \in \s,
\end{split}
\end{equation*}
with semigroup given by
\begin{equation*}
	\mathtt{P}^{\downarrow, (y, \Omega)}_t [f] (x,\nu) = \frac{1}{\mathtt{h}^\downarrow(x-y,\nu; \Omega)} {\bE}_{x,\nu}\left(f\left(\xi_t, \Theta_t\right)\mathtt{h}^\downarrow(\xi_t-y, \Theta_t;  \Omega)\one{t <  \tau_y^- }\right),
\end{equation*}
for any $x>y, t > 0$, $\nu \in \s.$ Recall that under Assumption~\ref{ass8} we can take $\Omega=\{\theta\}$ for any $\theta\in \s,$ and we use the notation ${\bP}_{x,\nu}^{\downarrow, y, \theta}$ instead of ${\bP}_{x,\nu}^{\downarrow, y, \{\theta\}},$ for the probabilities related to this process, similarly for the expectation and semigroup.

 We consider the finite dimensional distributions of the pre-global-minimum  of the path of $(\xi, \Theta)$ jointly  with the law of $(\xi_{g_\infty},\ubar{\Theta}_{g_\infty} )$, where we recall $g_\infty = \sup\{s>0: \underline\xi_s = \underline\xi_\infty\}$. Using that from \eqref{occupation_identity} the law of the pre-minimum is given by $\mathtt{U}_\nu^-(\dd y,\dd\theta)$, then given the pair $(\xi_{g_\infty},\Theta_{g_\infty} )$,  the law of the pre-minimum process possesses the required conditional behaviour. 

 \smallskip
Firstly, reversing time and applying the last exit formula yields, for some {non-negative}, uniformly bounded and measurable functionals $F, G, H,$ and $K:[0,\infty)\to \re,$ $x\in\R,$ $\Omega\subset\s$ measurable
and $t>0$, 
\begin{align*}
&\int_\s \pi(\dd \nu)H(\nu)\bE_{x,\nu} \left[ F \left( (\xi_s - \xi_{g_t}, \Theta_s) : 0 < s< g_t\right) G(\Theta_{g_t}) K(-\xi_{g_t}) \right]\notag\\
	&=\bE_{0,\pi} \left[ F \left( (\xi_s - \xi_{g_t}, \Theta_s) : 0 < s< g_t\right) G(\Theta_{g_t}) K(-\xi_{g_t}-x) H(\Theta_0) \right]
	\\
	&= \hat{\bE}_{0,\pi} \left[ F((\xi_{ t - u} - \xi_{g_t}, \Theta_{ t-u}) :0 < u < t-g_t ) G(\Theta_{g_t}) K(\xi_t-\xi_{g_t}-x)H(\Theta_t)\right]
	\\
	&= \hat{\bE}_{0,\pi} \left[ \sum_{g\in \underline{G}} F((\xi_{t-u} - \xi_{g}, \Theta_{t-u}) : 0< u < t-g) G(\Theta_{g}) K(\xi_t-\xi_{g}-x) H(\Theta_t)\one{ g < t < D_g } \right]
	\\
	&= \hat{\bE}_{0,\pi} \left[ \int_0^t \dd \ubar{L}_s G(\Theta_{s})  \hat{{\mathtt n}}_{\Theta_s}^- \left( F\left((\epsilon_{t-s-u},\Theta_{t-s-u}^\epsilon) : 0<u<t-s\right)K(\epsilon_{t-s}-x)  \one{t-s <\zeta} \right) \right]
	\\
	&= \int_\s\int_0^t \hat{U}_\pi^-(\dd \ell, \dd \theta) G(\theta)\\
	&\hspace{2cm} \times\hat{\mathtt n}_\theta^- \left( F( (\epsilon_{t-\ell-u}, \Theta_{t-\ell-u}^\epsilon): 0 < u < t - \ell) K(\epsilon_{t-\ell}-x)H(\Theta^\epsilon_{t-\ell})   \one{t - \ell <\zeta} \right).
\end{align*}
In what follows, we will work with functionals of the type
\begin{equation}
F((\omega_s, \eta_s) : 0 < s \leq T) = \prod_{i=1}^{n} f_i(\omega_{s_i} , \eta_{s_i}) \one{ s_n <T},
\label{specialF}
\end{equation}
for $0 < s_1 < ... < s_n $ and coordinate processes $(\omega_s, s\geq0)$ in $[0,\infty)$ and $(\eta_s, s\geq0)$ in $\mathcal{S}$ and $T>0$.
 Next, we may use the Markov property  to write the following iterative formula. Indeed, denoting by $\hat{\mathtt{P}}^\dagger=(\hat{\mathtt{P}}^\dagger_{t}, \ t\geq 0)$ the semigroup of the process killed at its first passage time below zero under $ \hat{\bP},$ we have
\begin{align}
	&\int_\s\int_0^t \hat{U}_\pi^-(\dd \ell, \dd \theta) G(\theta)  \hat{\mathtt n}_\theta^- \Big( K(\epsilon_{t-\ell}-x)H(\Theta^\epsilon_{t-\ell}) \prod_{i=1}^{n} f_{i}(\epsilon_{t-\ell-s_{i}}, \Theta_{t-\ell-s_{i}}^\epsilon)  \one{0 \leq s_1 < ... < s_n < t- \ell < \zeta}  \Big)
	\notag\\
	&=\int_\s\int_0^t \hat{U}_\pi^-(\dd \ell, \dd \theta) G(\theta)  \hat{\mathtt n}_\theta^- \Big(f_n(\epsilon_{t-\ell-s_n}, \Theta_{t-\ell-s_n}^\epsilon)\notag\\
	&\qquad\times  \hat{\mathtt{P}}^\dagger_{s_n-s_{n-1}} \left[f_{n-1}...\hat{\mathtt{P}}^\dagger_{s_2-s_{1}} \left[f_{1} \hat{\mathtt{P}}^\dagger_{s_1} [K(\cdot-x)H]  \right] \right] (\epsilon_{t-\ell-s_n}, \Theta_{t-\ell-s_n}^\epsilon)
	 \one{  0<t- \ell-s_n < \zeta} \Big).
	\notag\\
\end{align}
We now make $t\to\infty,$ and apply the Markov Renewal Theorem for MAPs to get that the above expression converges towards
\begin{align}	
	& \xrightarrow[t\to\infty]{}	\hat{\mathtt n}_{ G\hat{\mu}^-}^- \Bigg( \int_{0}^\infty \one{0< u-s_n < \zeta}  f_n(\epsilon_{u - s_n}, \Theta_{t-s_n}^\epsilon)   \notag\\
	&\hspace{3.3cm}\times\hat{\mathtt{P}}^\dagger_{s_n-s_{n-1}} \left[f_{n-1}...\hat{\mathtt{P}}^\dagger_{s_2-s_{1}} \left[f_{1} \hat{\mathtt{P}}^\dagger_{s_1} [K(\cdot-x)H] \right] \right] (\epsilon_{u-s_n}, \Theta_{u-s_n}^\epsilon)\dd u\Bigg)
	\notag\\
	&= \hat{\mathtt n}_{ G\hat{\mu}^-}^- \left( \int_{0}^\infty\one{v<\zeta}  f_n(\epsilon_{v}, \Theta^\epsilon_v) \hat{\mathtt{P}}^\dagger_{s_n-s_{n-1}} \left[f_{n-1}...\hat{\mathtt{P}}^\dagger_{s_2-s_{1}} \left[f_{1} \hat{\mathtt{P}}^\dagger_{s_1} [K(\cdot-x)H] \right]\right] (\epsilon_{v}, \Theta_{v}^\epsilon)\dd v\right).
	\label{verfiyMRTconditionsfirst}
\end{align}
where $ \hat{\mathtt n}_{ G\hat{\mu}^-}^-(\cdot)$ denotes the measure given by 
$\int_\mathcal{S}\hat{\mu}^-(\dd \theta) G(\theta)\hat{\mathtt n}_\theta(\cdot)$. In order to justify the application of the Markov Renewal Theorem above, we need to  verify some conditions, which we now attend to. Define the function $W_\theta$ by 
\begin{align*}
	W_\theta (t)  &=  \hat{\mathtt n}_\theta^- \Big(f_n(\epsilon_{t-s_n}, \Theta_{t-s_n}^\epsilon)\notag
	\\
	&\hspace{1cm}    \times\hat{\mathtt{P}}^\dagger_{s_n-s_{n-1}} \left[f_{n-1}...\hat{\mathtt{P}}^\dagger_{s_2-s_{1}} \left[f_{1} \hat{\mathtt{P}}^\dagger_{s_1} [K(\cdot-x)H]  \right] \right] (\epsilon_{t-s_n}, \Theta_{t-s_n}^\epsilon)
	\one{ 0<t-s_n < \zeta} \Big).
\end{align*}
Using the version of the Markov Renewal  Theorem  as stated in Theorem 1 of \cite{AlsmeyerTheMR}, it is sufficient to show the following integrability condition holds, for any $\Delta>0,$:
\begin{equation}
	\label{integrability_condition}
	\int_\s \sum_n \sup_{ \Delta n \leq l \leq \Delta ( n+1)} G(\theta) W_\theta (l)  \hat{\pi}^- (\dd \theta) < \infty.
\end{equation}
Suppose $0 \leq f_1, ... , f_{n}$ are continuous and bounded by some $M > 0.$ Then
\begin{equation*}
	W_\theta (l) \leq M^n \cdot \hat{\mathtt n}_\theta^- \left( \one{l - s_n < \infty }\right),
\end{equation*} 
and hence, arguing similarly to \eqref{MRTcondition}, we have
\begin{align*}
	\sum_m \sup_{ \Delta m \leq l \leq \Delta (m+1)} W_\theta(t) &\leq M^n  \sum_m \sup_{ \Delta m \leq l \leq \Delta ( m+1)} \hat{\mathtt n}_\theta^- \left( 0 < l-s_n < \zeta \right)
	\\
	&\leq M^n    \sum_m  \hat{\mathtt n}_\theta^- \left(\Delta m - s_n <\zeta \right) 
	\\
	& \leq  M^n    \sum_m \frac{1}{\Delta} \int_{\Delta (m-1)}^{\Delta m} \hat{\mathtt n}_\theta^- \left(\ell - s_n <\zeta  \right)\dd \ell
	\\
	&= \frac{M^n}{\Delta}  \int_{s_n}^\infty \hat{\mathtt n}_\theta^- \left(\ell - s_n <\zeta \right) \dd \ell, 
\end{align*}
which allows us to dominate \eqref{integrability_condition} by 
\begin{equation*}
	\int \hat{\pi}^- (\dd \theta) G(\theta)   \int_0^\infty \hat{\mathtt n}_\theta^- \left(\ell <\zeta \right) \dd \ell  = 
	\hat{\mathtt n}_{G\hat{\pi}^-}(\zeta)\leq \norm{G}_\infty \hat{\mathtt n}_{\hat{\pi}^-}(\zeta)< \infty,
\end{equation*}
thanks to Assumption \ref{ass5}. Returning to \eqref{verfiyMRTconditionsfirst}, we can now employ the occupation identity (\ref{occupation_identity}) to see that 
\begin{align*}
	&\hat{\mathtt n}_{ G\hat{\mu}^-}^- \left( \int_{0}^\infty \dd v f_n  \hat{\mathtt{P}}^\dagger_{s_n-s_{n-1}} \left[f_{n-1}...\hat{\mathtt{P}}^\dagger_{s_2-s_{1}} \left[f_{1} \hat{\mathtt{P}}^\dagger_{s_1} [K(\cdot-x)H] \right]\right] (\epsilon_{v}, \Theta_{v}^\epsilon)\one{v<\zeta}\right)
	\\
	&= \int_{\s \times [0,\infty) \times \s } \pi(\dd \nu)\mathtt{U}^-_\nu(\dd y , \dd \theta) G(\theta) \left( f_n  \hat{\mathtt{P}}^\dagger_{s_n-s_{n-1}} \left[f_{n-1}...\hat{\mathtt{P}}^\dagger_{s_2-s_{1}} \left[f_{1} \hat{\mathtt{P}}^\dagger_{s_1} [K(\cdot-x)H] \right]\right] (y,\nu) \right)
	\\
	&= \int_{\s \times [0,\infty) \times \s }  \pi(\dd \nu)\dd y\Xi(\dd\theta) \mathtt{u}^-_\nu( y , \theta) G(\theta)\left( f_n  \hat{\mathtt{P}}^\dagger_{s_n-s_{n-1}} \left[f_{n-1}...\hat{\mathtt{P}}^\dagger_{s_2-s_{1}} \left[f_{1} \hat{\mathtt{P}}^\dagger_{s_1} [K(\cdot-x)H] \right]\right] (y,\nu) \right)
	\\
	&= \int_{\s \times [0,\infty)\times \s} \pi(\dd\nu) \dd y \Xi(\dd\theta) \mathtt{h}^\downarrow (y ,\nu; \theta)G(\theta)\\
	&\quad\times\left( f_n  \hat{\mathtt{P}}^{\dagger}_{s_n - s_{n-1}} \left[ f_{n-1}\hat{\mathtt{P}}^{\dagger}_{s_{n-1}-s_{n-2}} \left[...f_1 \hat{\mathtt{P}}^{\dagger}_{s_{1}} [K(\cdot-x)H] \right]\right] (y,\nu)\right);\\
\end{align*}
where as before $\mathtt{h}^\downarrow (y ,\nu; G)$ is as defined in \eqref{hdownarrow2}. Finally, using duality and introducing the function $\mathtt{h}^\downarrow$ in a telescopic way, we get 
\begin{equation}\label{needsinterpreting}
\begin{split}
&\int_{\s\times\s \times [0,\infty)} \Xi(\dd\theta)\pi(\dd\nu) \dd y   \mathtt{h}^\downarrow (y ,\nu; \theta)G(\theta)\\
&\quad\times\left( f_n  \hat{\mathtt{P}}^{\dagger}_{s_n - s_{n-1}} \left[ f_{n-1}\hat{\mathtt{P}}^{\dagger}_{s_{n-1}-s_{n-2}} \left[...f_1 \hat{\mathtt{P}}^{\dagger}_{s_{1}} [K(\cdot-x)H] \right]\right] (y,\nu)\right)\notag \\
&= \int_{\s\times\s \times [0,\infty)} \Xi(\dd\theta) \pi(\dd \nu) dy K(y-x)H(\nu) G(\theta) {\mathtt{P}}^{\dagger}_{s_1} \left[ {f}_1 {\mathtt{P}}^{\dagger}_{s_{2}-s_{1}} \left[... {\mathtt{P}}^{\dagger}_{s_n-s_{n-1}} [\mathtt{h}^\downarrow f_n ] \right]\right] (y,\nu)
\notag	\\
&=\int_{\s\times\s \times [x,\infty)}  \Xi(\dd \theta) \pi(\dd \nu) dy K(y-x)H(\nu)G(\theta)\\
&\qquad\times \mathtt{h}^\downarrow (y ,\nu; \theta) {\mathtt{P}}^{\downarrow, (0,\theta)}_{s_1} \left[ {f}_1 {\mathtt{P}}^{\downarrow, (0,\theta)}_{s_{2}-s_{1}} \left[... {\mathtt{P}}^{\downarrow, (0,\theta)}_{s_n-s_{n-1}} [ f_n ] \right]\right] (y,\nu);  \notag \\
\end{split}
\end{equation}
next, we use the the Markov property, the fact that $K$ has support on $[0,\infty)$, and a change of variables to show that the above expression equals
\begin{equation}
\begin{split}
&= \int_{\s\times\s \times [x,\infty)}  \Xi(\dd \nu)\pi(\dd \nu) dy K(y-x)H(\nu)G(\theta) {\mathtt{h}^\downarrow (y ,\nu; \theta)}  \bE_{y,\nu}^{\downarrow, (0,\theta)} \left[  \prod_{i=1}^{n} f_i(\xi_{s_i} , \Theta_{s_i}), s_n<\zeta  \right] \notag\\
&= \int_{\s\times\s \times [x,\infty)} \Xi(\dd \theta) \pi(\dd \nu) dy K(y-x)H(\nu) \\
&\qquad\times \bE_{0,\nu}^{} \left[  \prod_{i=1}^{n} f_i(\xi_{s_i}+y, \Theta_{s_i}) {\mathtt{h}^\downarrow (\xi_{s_n}+y ,\Theta_{s_n}; \theta)}, s_n<\tau^{-}_{-y}  \right] \notag\\
&= \int_{\s \times\s\times [0,\infty)}  \Xi(\dd \theta)\pi(\dd \nu) dz K(z)H(\nu) G(\theta)\\ &\qquad\times \bE_{0,\nu} \left[  \prod_{i=1}^{n} f_i(\xi_{s_i}+z+x, \Theta_{s_i}) {\mathtt{h}^\downarrow (\xi_{s_n}+z+x ,\Theta_{s_n}; \theta)}, s_n<\tau^{-}_{-(z+x)}  \right] ;\notag\\
\end{split}
\end{equation}from this expression, and recalling how  $\bE_{\cdot,\cdot}^{\downarrow, (0,\theta)}$ is constructed, we obtain now that the latter equals 
\begin{equation}
\begin{split}
&= \int_{\s\times\s \times [0,\infty)}  \Xi(\dd \theta) \pi(\dd \nu) dz K(z)H(\nu)G(\theta)  \frac{\mathtt{h}^\downarrow (z+x ,\nu; \theta)}{\mathtt{h}^\downarrow (z+x ,\nu; \theta)}\\
&\qquad\times\bE_{x+z,\nu}^{} \left[  \prod_{i=1}^{n} f_i(\xi_{s_i}, \Theta_{s_i}) {\mathtt{h}^\downarrow (\xi_{s_n},\Theta_{s_n}; G)}, s_n<\tau^{-}_{0}  \right] \notag\\
&=\int_{\s\times\s \times [0,\infty)}   \Xi(\dd \theta)\pi(\dd \nu) dz K(z)H(\nu) {\mathtt{h}^\downarrow (z+x,\nu; \theta)} G(\theta) \bE_{x+z,\nu}^{\downarrow, (0,\theta)} \left[ F((\xi_{s} , \Theta_{s}) : 0 < s <\zeta)  \right]
\notag	\\	
	&= \int_{\s\times\s\times [0,\infty) } \pi(\dd \nu) H(\nu)\bP_{x,\nu} (-\ubar{\xi}_\infty \in \dd z, \ubar{\Theta}_\infty \in \dd \theta)\\
	&\qquad\times G(\theta)K(z) \bE_{x+z,\nu}^{\downarrow, (0, \theta)} \left[ F((\xi_{s} , \Theta_{s}) : 0 < s < \zeta)  \right],		
\end{split}
\end{equation}
where $\zeta$ denotes the lifetime of the process under $\bP^{\downarrow, 0,\theta}$ and we have used \eqref{MAPPOCR}.

\smallskip

It follows that for $\pi$-almost every $\nu\in \s$,
\begin{align}
&\bE_{x,\nu} \left[ F \left( (\xi_s - \xi_{g_\infty}, \Theta_s) : 0 < s< g_{\infty}\right) G(\Theta_{g_\infty}) K(-\xi_{g_\infty}) \right]\notag\\
&=\int_{[0,\infty) \times \s } \bP_{x,\nu} (-\ubar{\xi}_\infty \in \dd y, \ubar{\Theta}_\infty \in \dd \theta) G(\theta)K(y) \bE_{y+x,\nu}^{\downarrow, (0, \theta)} \left[ F((\xi_{s} , \Theta_{s}) : 0 < s < \zeta)  \right].
\label{needsinterpreting21}
\end{align}
{The continuity property in Assumption~\ref{ass8} is enough to ensure  that \eqref{needsinterpreting21} holds for all $\nu\in \s$ and $x\in\R.$}
\smallskip
In conclusión, we have that for any $F$ as defined in \eqref{specialF}, we have
\begin{align}
&\bE_{x,\nu} \left[ F \left( (\xi_s , \Theta_s) : 0 < s< g_{\infty}\right) G(\Theta_{g_\infty}) K(-\xi_{g_\infty}) \right]\notag\\
&=\int_{[0,\infty) \times \s } \bP_{x,\nu} (-\ubar{\xi}_\infty \in \dd z, \ubar{\Theta}_\infty \in \dd \theta) G(\theta)K(z) \bE_{x+z,\nu}^{\downarrow, (0,\theta)} \left[ F((\xi_{s} -z, \Theta_{s}) : 0 < s < \zeta)  \right]\notag\\
&=\int_{[0,\infty) \times \s } \bP_{x,\nu} (-\ubar{\xi}_\infty \in \dd z, \ubar{\Theta}_\infty \in \dd \theta) G(\theta)K(z) \bE_{x,\nu}^{\downarrow, (-z,\theta)} \left[ F((\xi_{s} , \Theta_{s}) : 0 < s < \zeta)  \right].
\label{needsinterpreting2}
\end{align}
Since the functions $F$ as defined in \eqref{specialF} determine the finite dimensional distributions, this is enough to ensure that \eqref{needsinterpreting2} holds for any measurable and bounded functional $F$  of the pre-minimum process. Thus, under $\bE_{x,\nu},$ the conditional law of the pre-minimum process $\left( (\xi_s , \Theta_s) : 0 < s< g_{\infty}\right)$ given $(\Theta_{g_\infty}, \xi_{g_\infty})=(\theta,z),$ equals that of $((\xi_{s} , \Theta_{s}) : 0 < s < \zeta)$ under $\bE_{x,\nu}^{\downarrow, (-z,\theta)}.$

\smallskip

Next, we turn our attention to the post-minimum process, starting with an application of the last exit formula. Let us denote by $\bP^{\uparrow,(w)}_{x,\theta}$ the law of the process conditioned to avoid the sphere of radius $w\geq 0,$ with probability measure  defined by 
 \begin{equation*}
 \begin{split}
 \bE^{\uparrow, (w)}_{x,\theta}\left(F(\xi_s, \Theta_s, s\geq 0)\right)&=\bE_{x,\theta}\left(F(\xi_s, \Theta_s, s\geq 0) | \tau^{-}_{w}=\infty \right)\\
&=\bE_{x-w,\theta}\left(F(\xi_s+w, \Theta_s, s\geq 0) | \tau^{-}_{0}=\infty \right) 
 ,\qquad x>w, \theta\in \s;
 \end{split}
 \end{equation*}
and by 
\begin{equation*}
 \begin{split}
 \bE^{\uparrow, (w)}_{w,\theta}\left(F(\xi_s, \Theta_s, s\geq 0)\right)={\mathtt n}^-_{\theta}\left(F(\xi_s+w, \Theta_s, s\geq 0) | \zeta=\infty \right), \qquad\theta\in \s;
 \end{split}
 \end{equation*}
for any $F$ measurable and bounded functional of $(\xi,\Theta).$ Recall that by assumption~\eqref{ass4} the event $\{\tau^{-}_{w}=\infty\}$ has strictly positive probability under $\bP_{x,\theta},$ for any $x>w, \theta\in \s;$ as well as ${\mathtt n}^-_{\theta}(\zeta=\infty)>0,$ for all $\theta\in \s.$ It follows from the definition that for any given $w\geq0,$ under $\bE^{\uparrow, (w)}$ the coordinate process has the strong Markov property and has transition semigroup given by $$\mathtt{P}^{\uparrow,(w)}_{r}f(x,\theta):=\begin{cases}\bE_{x,\theta}\left(f(\xi_r, \Theta_r) | \tau^{-}_{w}=\infty\right),& x>w,\\
{\mathtt n}^-_{\theta}\left(f(\xi_r+w, \Theta_r) | \zeta=\infty \right),& x=w, \end{cases} \qquad \theta\in\s, r\geq 0; 
$$ where $f:\R\times\s\mapsto \R^{+}$ is any measurable and positive function. Furthermore, for $x>w\geq 0,$ $\bP^{\uparrow, (w)}_{x}$ is obtained as an $h$-transform and we have the following change of measure,
\begin{equation}
	\left.	\frac{\dd \bP^{\uparrow, (w)}_{x,\theta}}{\dd \bP_{x,\theta}}\right|_{\mathcal{F}_t \cap \{t < \zeta\}}  = \frac{\mathtt{h}^\uparrow(\xi_t-w, \Theta_t)}{\mathtt{h}^\uparrow(x-w,\theta)} \one{t < \tau_{w}^-}, \qquad x > w,\theta\in \mathcal{S},
	\label{xithetatoavoidsphere}
\end{equation}
while for $x=w\geq 0,$
\begin{equation}
\begin{split}
	\bE_{w,\theta}^{\uparrow,(w)} \left[ F((\xi_s, \Theta_s), s\leq t), t<\zeta \right]  &= {\mathtt n}^-_\theta \left( \frac{\mathtt{h}^\uparrow(\xi_t, \Theta_t)}{\mathtt{h}^\uparrow(0,\theta)}  F((\xi_s+w, \Theta_s), s\leq t) \one{t<\zeta} \right)\\
	&= \bE_{0,\theta}^{\uparrow,(0)} \left[ F((\xi_s+w, \Theta_s), s\leq t), t<\zeta \right]
	\label{choiceat0}
\end{split}
\end{equation}
where $$\mathtt{h}^{\uparrow}(z,\theta) : = \begin{cases}\bP_{z,\theta}( \tau_0^- = \infty), & z>0, \theta\in \s,\\
{\mathtt n}^-_\theta (\zeta=\infty), & z=0, \theta\in \s.\end{cases}$$  
Observe that under $\bE^{\uparrow,(w)}$ the coordinate process has infinite lifetime a.s. $$\bP_{x,\theta}^{\uparrow,(w)}(\zeta=\infty)=1,\qquad x\geq w, \theta\in\s.$$

Now, recalling that
\[
\mathtt{H}^\uparrow (x) = \mathtt{h}^\uparrow(\log \norm{x}, \arg(x)), \qquad \norm{x}\geq 1,
\]
 the Optimal Stopping Theorem, see e.g. Theorem 3.4 in \cite{jacod2013limit}, allow us to translate \eqref{xithetatoavoidsphere} into \eqref{doobup} and \eqref{choiceat0} into \eqref{doobup0}.

\smallskip
%
%

We will next describe the post-infimum process and relate it to the one with law $\bP^{\uparrow, \cdot}.$ To which end we use the last exit decomposition for non-negative, bounded and measurable $F$, $G$ and $\nu\in \s$, yielding
\begin{align*}
	&\bE_{x,\nu} \left[ F( (\xi_{g_\infty + t} - \xi_{g_\infty}, \Theta_{g_\infty +t}) : t \geq 0) G(-\xi_{g_\infty} , \Theta_{g_\infty}) \right]
	\\
	&= \bE_{0,\nu} \left[ \sum_{g\in \underline{G}} F((\xi_{g + t} - \xi_{g}, \Theta_{g +t}) : t \geq 0) G(-\xi_{g}-x , \Theta_{g}) \one{\zeta_g = \infty}\right]
	\\
	&= \bE_{0,\nu} \left[ \int_0^\infty \dd \ubar{L}_s G(-\xi_{s} -x, \Theta_{s}) {\mathtt n}_{\Theta_s}^- \left( F((\epsilon_t, \Theta_t^\epsilon) : t \geq 0) \one{\zeta = \infty }\right) \right]
	\\
	&=\bE_{0,\nu} \left[ \int_0^\infty \dd t  G(\xi_{s}^- -x, \Theta_{s}^-) {\mathtt n}_{\Theta_s^-}^- \left( F((\epsilon_t, \Theta_t^\epsilon) : t \geq 0) \one{\zeta = \infty }\right) \right]
	\\
	&= \int_{[0,\infty) \times \s} U_\nu^-(\dd y , \dd \theta) G(y-x,\theta) {\mathtt n}_{\theta}^- \left( F((\epsilon_t, \Theta_t^\epsilon) : t \geq 0) \one{\zeta = \infty }\right).
\end{align*}
Using Assumption \ref{ass8}, we may rewrite this in terms of the density $\mathtt{u}^-,$ paying close attention to the normalising factor ${\mathtt n}_\theta^-(\zeta = \infty)$ that it will bring with it. Indeed,
\begin{align*}
	& \int_{[0,\infty) \times \s} U_\nu^-(\dd y , \dd \theta) G(y-x,\theta) {\mathtt n}_{\theta}^- \left( F((\epsilon_t, \Theta_t^\epsilon) : t \geq 0) \one{\zeta = \infty }\right)
	\\
	&= \int_{[0,\infty) \times \s} \Xi(\dd \theta) \dd y \mathtt{u}_\nu^-(y,\theta) G(y-x,\theta) \frac{ {\mathtt n}_{\theta}^- \left( F((\epsilon_t, \Theta_t^\epsilon) : t \geq 0 ) \one{\zeta = \infty } \right)}{{\mathtt n}_\theta^-(\zeta = \infty) }
	\\
	&= \int_{[0,\infty) \times \s} \Xi(\dd \theta) \dd y \mathtt{u}_\nu^-(y,\theta) G(y-x,\theta) \frac{ {\mathtt n}_{\theta}^- \left( F((\epsilon_t, \Theta_t^\epsilon) : t \geq 0 ) \one{\zeta = \infty } \right)}{{\mathtt n}_\theta^-(\zeta = \infty) }.
\end{align*}
{Continuing the computation with the special choice of $F$ given in \eqref{specialF}, we can 
use the Markov property of the transition semigroup $\mathtt{P}^\dagger,$ which yields
\begin{align*}
	& \int_{[0,\infty) \times \s} \Xi(\dd \theta) \dd y \mathtt{u}_\nu^-(y,\theta) G(y-x,\theta) \frac{ {\mathtt n}_{\theta}^- \left( F((\epsilon_t, \Theta_t^\epsilon) : t \geq 0 ) \one{\zeta = \infty } \right)}{{\mathtt n}_\theta^-(\zeta = \infty) } 
	\\
	&= \int_{[0,\infty) \times \s} \Xi(\dd \theta) \dd y \mathtt{u}_\nu^-(y,\theta) G(y-x,\theta) \frac{ {\mathtt n}_{\theta}^- \left( \prod_{i=1}^n f_i(\epsilon_{s_i}, \Theta_{s_i}) \mathtt{h}^\uparrow(\epsilon_{s_n}, \Theta_{s_n}) \right)}{{\mathtt n}_\theta^-(\zeta = \infty) } 
	\\
	&= \int_{[0,\infty) \times \s} \Xi(\dd \theta) \dd y \mathtt{u}_\nu^-(y,\theta) G(y-x,\theta) \frac{ {\mathtt n}_{\theta}^- \left( f_1(\epsilon_{s_1}, \Theta_{s_1}) \mathtt{P}^\dagger_{s_1}\left[f_2 \mathtt{P}^\dagger_{s_2 - s_1} ... \mathtt{P}^\dagger_{s_{n} - s_{n-1}} \left[  f_n\mathtt{h}^\uparrow\right]   \right]\right)}{{\mathtt n}_\theta^-(\zeta = \infty) }
	\\
	&= \int_{[0,\infty) \times \s} \Xi(\dd \theta) \dd y \mathtt{u}_\nu^-(y,\theta) G(y-x,\theta)\notag\\
	&\hspace{3cm} \frac{ {\mathtt n}_{\theta}^- \left( \mathtt{h}^\uparrow(\epsilon_{s_1}, \Theta_{s_1})f_1(\epsilon_{s_1}, \Theta_{s_1}) \mathtt{P}^{\uparrow,(0)}_{s_1}\left[f_2 \mathtt{P}^{\uparrow,(0)}_{s_2 - s_1} ... \mathtt{P}^{\uparrow,(0)}_{s_{n} - s_{n-1}} \left[  f_n\right]   \right](\epsilon_{s_1}, \Theta_{s_1}) \right)}{{\mathtt n}_\theta^-(\zeta = \infty) }\notag\\
	&= \int_{[0,\infty) \times \s} \Xi(\dd \theta) \dd y \mathtt{u}_\nu^-(y,\theta) G(y-x,\theta)\mathtt{P}^{\uparrow,(0)}_{s_1} \left[f_1 \mathtt{P}^{\uparrow,(0)}_{s_2-s_1}\left[f_2 \mathtt{P}^{\uparrow,(0)}_{s_3 - s_1}\left[ ... \mathtt{P}^{\uparrow,(0)}_{s_{n} - s_{n-1}} \left[  f_n\right] \right]    \right] \right](0,\theta)
	\notag\\
	&= \int_{[0,\infty) \times \s} \Xi(\dd \theta) \dd y \mathtt{u}_\nu^-(y,\theta) G(y-x,\theta) \mathbf{E}^{\uparrow,0}_{0, \theta}[F((\xi_s, \Theta_s), s\geq 0)].
\end{align*}
Using \eqref{choiceat0}, we conclude that
\begin{align*}
	&\bE_{x,\nu} \left[ F( (\xi_{g_\infty + t}, \Theta_{g_\infty +t}) : t \geq 0) G(-\xi_{g_\infty} , \Theta_{g_\infty}) \right]\\
	&=\int_{[0,\infty) \times \s } \bP_{x,\nu} (-\ubar{\xi}_\infty \in \dd z, \ubar{\Theta}_\infty \in \dd \theta)G(z,\theta)\mathbf{E}^{\uparrow,(0)}_{0, \theta}[F((\xi_s-z, \Theta_s), s\geq 0)]\\
	&=\int_{[0,\infty) \times \s } \bP_{x,\nu} (-\ubar{\xi}_\infty \in \dd z, \ubar{\Theta}_\infty \in \dd \theta)G(z,\theta)\mathbf{E}^{\uparrow,(-z)}_{-z, \theta}[F((\xi_s, \Theta_s), s\geq 0)].
\end{align*}
The monotone class theorem allows us to extend the above identity, from the finite dimensional distributions to path functionals, and therefrom obtain the result for ssMp. 

To conclude, all that remains is to justify the independence of the pre-infimum and post-infimum given the infimum. However, a further application of the last exit decomposition followed by calculations similar to those performed in the preceeding lines allows us to see that for all bounded and measurable functionals $H,F,$ and $G,$
\begin{align*}
	&\bE_{x,\nu} \left[ H(\xi_s, s\leq g_{\infty})G(-\xi_{g_\infty} , \Theta_{g_\infty})F( (\xi_{g_\infty + t}, \Theta_{g_\infty +t}) : t \geq 0)  \right]\\
	&=\int_{[0,\infty) \times \s } \bP_{x,\nu} (-\ubar{\xi}_\infty \in \dd z, \ubar{\Theta}_\infty \in \dd \theta)\bE_{x,\nu}^{\downarrow, (-z,\theta)} \left[ H((\xi_{s} , \Theta_{s}) : 0 < s < \zeta)  \right]G(z,\theta)\\ 
	&\qquad \times\mathbf{E}^{\uparrow,(-z)}_{-z, \theta}[F((\xi_s, \Theta_s), s\geq 0)],
\end{align*}
from which the conditional independence follows. It is worth noting that alternatively we could have applied Theorem 6 from \cite{millar1978}  to conclude that the post and pre minimum paths are independent, conditionally on the point of closest reach, but we have chosen to provide the proof directly for the sake of completeness. \hfill$\square$

\smallskip


}

\section{Isotropic stable L\'evy processes}\label{Isostable}
A case of immediate interest to us is the isotropic stable L\'evy process. 
Let $X=(X_t, t \geq 0)$ be a $d$-dimensional  stable  L\'evy process $(d\geq 2)$ with probabilities $(\mathbb{P}_x, x \in \mathbb{R}^d)$. This means that $X$ has c\`adl\`ag paths with stationary and independent increments as well as there existing an  $\alpha>0$ such that, for $c>0,$ and $x \in \mathbb{R}^d,$
under $\mathbb{P}_x$,
\[
\text{ the law of } (cX_{c^{-\alpha}t}, t \geq 0) \text{ is equal to  }\mathbb{P}_{cx}. 
\]
The latter is the property of so-called self-similarity.
It turns out that stable L\'evy processes necessarily have  $\alpha\in (0,2]$. The case $\alpha=2$ is that of   standard $d$-dimensional Brownian motion, thus has a continuous path. All other $\alpha\in(0,2)$ have no Gaussian component and are  pure jump processes. In this article we are specifically interested in phenomena that can only occur when jumps are present. We thus restrict ourselves henceforth to the setting $\alpha\in(0,2)$. 

\smallskip

Although Brownian motion is isotropic, this need not be the case in the stable case when $\alpha\in(0,2)$. {\it Nonetheless, we will restrict  to the  isotropic setting.} To be more precise, this means, for all orthogonal transformations $U:\mathbb{R}^d \mapsto \mathbb{R}^d$ and $x \in \mathbb{R}^d,$ 
\[
\quad \textit{the law of} \quad (UX_t, t \geq 0) \textit{ under} \quad \mathbb{P}_x  \textit{ is equal to }  (X_t,t\geq 0) \textit{ under } \mathbb{P}_{Ux}.
\]
For convenience, we will henceforth refer to $X$  as a {\it stable process}. 

\smallskip

As a L\'evy process, our stable   process of index $(0,2)$ has a characteristic triplet $(0,0,\Pi)$, where the jump measure $\Pi$ satisfies 
\begin{equation}
	\Pi(B) = \frac{2^{\alpha} \Gamma((d+\alpha)/2)}{\pi^{d/2} |\Gamma(-\alpha/2)|} \int_B \frac{1}{|y|^{\alpha+d}} \ell_d({\rm d} y), \quad B \subseteq \mathcal{B}(\mathbb{R}^d),
	\label{bigPi}
\end{equation}
where $\ell_d$ is $d$-dimensional Lebesgue measure\footnote{We will distinguish integrals with respect to one-dimensional Lebesgue measure as taking the form $\int\cdot\, \dd x$, where as higher dimensional integrals will always indicate the dimension, for example $\int \cdot\, \ell_d(\dd x)$.}. 
This is equivalent to identifying its characteristic exponent as 
\[
\Psi(\theta)=-\frac{1}{t}\log \mathbb{P} ({\rm e}^{\iu\theta \cdot X_t}) =|\theta|^{\alpha}, \quad \theta \in \mathbb{R}^d,
\]
where we write $\mathbb{P}$ in preference to $\mathbb{P}_0$.

\smallskip

Assumptions \ref{ass1} automatically holds, Assumption \ref{ass2} holds with $\pi$ given by the uniform measure on $\mathbb{S}^{d-1}$, see e.g. \cite{TAMS} Section 7.2, paragraph (a2). That Assumption \ref{ass3} holds is a consequence of the fact that $(\norm{X_t}, t\geq0)$ is a positive self-similar Markov process whose underlying L\'evy process is a Lamperti stable process, which is regular for the upper and lower half lines. Assumption \ref{ass4} also automatically holds on account of the fact that $\lim_{t\to\infty}\norm{X_t}=\infty$.

\smallskip

In Kyprianou et al. \cite{kyprianou2018deep} it was shown that the point of closest reach to the origin, as described in \eqref{x*}, can be described in closed form in Cartesian coordinates. Indeed, 
\begin{equation}
\mathbb{P}_x(X^*\in \dd  y)=\pi^{-d/2}\dfrac{\Gamma\left({d}/{2}\right)^2}{\Gamma\left(({d-\alpha})/{2}\right)\Gamma\left({\alpha}/{2}\right)}
\, \frac{(\norm{x}^2-\norm{y}^2)^{\alpha/2}}{\norm{x-y}^{d}\norm{y}^{\alpha}}{\rm d} y,\qquad 0<\norm{y}<\norm{x}.
\label{POCR}
\end{equation}

Similarly, in Kyprianou et al. \cite{KPS}, it was also shown that the Doob $h$-transform to condition the stable process to hit  a patch $\Omega$  on $\mathbb{S}^{d}(1)$ is constructed as per Theorem \ref{condition1}, using the function 
\begin{equation}
\mathtt{H}^\downarrow(x; \Omega) = \left\{\begin{array}{ll}
 |\norm{x}^2-1|^{\alpha/2}\dint_\Omega \norm{\theta-x}^{-d} \sigma_1({\rm d}\theta)& \text{ if }\sigma_1(\Omega)>0,\\
&\\
|\norm{x}^2-1|^{\alpha/2}\norm{x-\arg(\vartheta)}^{-d}  &\text{ if }\Omega =\{\vartheta\},
\end{array}
\right.
\label{H_S}
\end{equation}
where $\sigma_1$ is the uniform surface measure on $\mathbb{S}^{d-1}$ with unit total mass and $\norm{x}>1$. Also \cite{KPS} the function $\mathtt{H}^\uparrow$ was identified, again in Cartesian coordinates, by 
\begin{equation}
 \mathtt{H}^\uparrow(x)= \int_0^{\norm{x}^2-1} (u+1)^{-d/2} u^{\alpha/2-1} \dd u, \qquad \norm{x}>1.
    \label{z=oo}
\end{equation}
Note that both in the case of \eqref{H_S} and \eqref{z=oo}, the functions are defined up to an unimportant multiplicative constant. 

\smallskip

Isometry for the the stable process forces $\pi(\dd \theta) = \hat{\pi}^-(\dd \theta) = \sigma_1(\dd \theta)$, for $\theta\in \mathcal{S} = \mathbb{S}^{d-1}$. Isometry also implies that ${\mathtt n}^-_\theta(\zeta = \infty)$ is invariant of $\theta\in\mathbb{S}^{d-1}$ and hence can be treated as a constant. Isotropy also means that $\hat{\mathtt n}_{\theta}(\zeta)$ does not depend on $\theta$. We also note that $\hat{\mathtt n}_{\theta}(\zeta)$ equals the mean excursion time under the excursion measure of $\xi$ from its maximum, say $n^+_{\xi}$.  Since $\mathbf{P}[|\xi_1|]<\infty$, 
standard Wiener--Hopf theory for L\'evy processes ensures that $n^+_\xi(\zeta)<\infty$; see e.g. Chapter 2 of Kyprianou and Pardo \cite{ssmpbook}. 
Hence Assumption \ref{ass5} is also satisfied.

\smallskip

Turning to Assumptions \ref{ass7} and \ref{ass8}, by comparing \eqref{H_S} and \eqref{POCR} with \eqref{MAPPOCR}, \eqref{MAPPOCR2}  and \eqref{hdownarrow2}, we see that, up to a multiplicative constant,  from \eqref{POCR}, with $\norm{x} = 1$, we deduce
\begin{equation}
\mathtt{u}^-_\theta(u,\varphi) = (1-{\rm e}^{-2u})^{\alpha/2}\norm{\theta-{\rm e}^{-u}\varphi}^{-d}{\rm e}^{(\alpha-d) u} , \qquad 
u\geq0, \varphi\in\mathbb{S}^{d-1},
\label{usesk}
\end{equation} 
which can be re-organised to agree with \eqref{z=oo}, noting that $H^\downarrow(x; \{\vartheta\}) = \mathtt{u}^-_{\arg(x)}(\log\norm{x}, \varphi)$.

\begin{remark}
\rm
Although the above example pertains to stable L\'evy processes, supported by the work of \cite{ssmpbook} and \cite{KPS}, the reader can easily verify from these publications that the case $\alpha = 2 < d$ is equally valid. That is, the case of transient Brownian motion (i.e. in dimension three and greater).
\end{remark}

\section{Planar Brownian motion in a cone}\label{BMcone}

In this example we take $(X, \mathbb{P})$ to be an isotropic two-dimensional standard Brownian motion. It is known that we can write 
\begin{equation}
X_t = {\rm e}^{\xi_{\varphi_t}+{\rm i}\theta_{\varphi_t}}, \qquad 0\leq t\leq T_0,
\label{planar}
\end{equation}
where $(\xi_t, t\geq0)$ is a Brownian motion, which is independent of the winding number $(\theta_t, t\geq0)$
which is also a Brownian motion, $T_0 = \inf\{t>0: X_t = 0\}$ and,  as usual, $\varphi_t = \inf\{s>0: \int_0^s\exp(2\xi_u)\dd u>t\}$. See for example p194 of \cite{RY}. We can see from \eqref{planar} that the Lamperti--Kiu decomposition in \eqref{eq:lamperti_kiu} dictates that the pair $(\xi,\exp({\rm i}\theta) )$ is a MAP. This is otherwise trivial to observe.  In this special planar setting, we prefer to refer to the MAP as the pair $(\xi, \theta)$ instead with probabilities $\mathbb{P} = (\mathbf{P}_{x,\phi}, x\in\mathbb{R}, \phi\in (-\pi,\pi])$. 

\smallskip

We are interested in the construction of $X$ conditioned to stay in a cone. We can describe the cone in $\mathbb{R}^2$ through polar coordinates $0\leq r<\infty$ and $\phi\in(-\pi,\pi]$. Fix $\phi_0\in(0,\pi)$, thanks to the isotropy of $X$, it suffices to consider a cone  $\Gamma: = \{(r,\phi): 0<r<\infty, \, \phi\in(-\phi_0, \phi_0)\}$ described in polar coordinates.

\smallskip

Let us start by considering the two-dimensional Brownian motion killed on exiting $\Gamma$, that is,  $X^\Gamma_t: = X_t \one{t<T_\Gamma}$, $t\geq0,$ where $T^\Gamma = \inf\{t>0: X_t\not\in \Gamma\}$. It is known that $X^\Gamma$ is a self-similar Markov process, and, for convenience,  we write $(\xi^\Gamma, \theta^\Gamma)$ for its underlying MAP, where, as in \eqref{planar}, we write $\theta^\Gamma$ for the winding number.

\smallskip

 It is worth noting that, because of the winding properties of Brownian motion, it will  necessarily exit from the sides of the cone, rather than its apex at the origin.  To see this, note that $X$ exists $\Gamma$ if and only if the pair of processes $(\xi,\theta)$ exit the strip $\mathbb{R}\times (-\phi_0,\phi_0)$. Accordingly, the event that $X$ exits the apex of $\Gamma$ corresponds to the event that $\xi$ reaches $-\infty$ before $\theta$ exists $(-\phi_0,\phi_0)$, which occurs with probability 0.

\smallskip 

Next,  in the spirit of \eqref{MAPPOCR2}, we are interested in the potential of the descending ladder MAP of the pair $(\xi^\Gamma, \theta^\Gamma)$ corresponding to the setting that  $X$ is killed on exiting $\Gamma$. More precisely, for each $\phi\in(-\phi_0, \phi_0)$, let us define the aforesaid potential via the relation
\begin{equation}
\begin{split}
&\int_0^\infty\int_0^\infty \int_{(-\phi_0,\phi_0)} f(s, y, \theta)U^{-,\dagger}_\phi(\dd s, \dd y, \dd \theta)\\ 
&= \mathbf{E}_{0, \phi}\left[\int_0^\infty 
f(\underline{L}^{-1}_t, \xi_{\underline{L}^{-1}_t}, \theta_{\underline{L}^{-1}_t})\one{\underline{L}^{-1}_t < T^\Gamma}\dd t \right],
\label{UGamma}
\end{split}
\end{equation}
for bounded measurable functions $f: [0,\infty)\times \mathbb{R}^2\to [0,\infty)$. Here, $\underline{L}$ is the local time of $\xi$ at its minimum under $\mathbb{P}$. Thanks to path continuity, we have that $\xi_{\underline{L}^{-1}_t} = t$. Moreover, since $\xi$ and $\theta$ are independent, we can further develop the right-hand side of \eqref{UGamma}, 
\begin{align}
\int_0^\infty\int_0^\infty \int_{(-\phi_0,\phi_0)} f(s, y, \theta)U^{-,\dagger}_\phi(\dd s, \dd y, \dd \theta) &= 
\mathbf{E}_{0, \phi}\left[\int_0^\infty 
H_\phi^f(\underline{L}^{-1}_y, y)\dd y \right]\notag\\
&=
\int_0^\infty \dd y \int_0^\infty \mathbf{P}_{0,\phi}(\underline{L}^{-1}_y\in \dd s)H_\phi^f(s, y),
\label{needsbmsg}
\end{align}
where 
\[
H_\phi^f(s, y) = \mathbf{E}_{0,\phi}\left[f(s, y, \theta_s)\one{s<\tau^{(-\phi_0,\phi_0)}}\right],
\]
and $\tau^{(-\phi_0,\phi_0)} = \inf\{t>0: \theta_t \not\in(-\phi_0,\phi_0)\}$. 

\smallskip

On account of the fact that $\theta$ is also a Brownian motion, it is well known that the semigroup of a Brownian motion killed on exiting the interval $(-\phi_0,\phi_0)$ is given by 
\begin{equation}
\begin{split}
&\mathbf{P}_{0,\phi} (\theta_s\in \dd \theta, \, s<\tau^{(-\phi_0,\phi_0)})\\ 
&= \frac{1}{\phi_0}\sum_{k =1}^\infty{\rm e}^{- \pi^2 k^2 s/8\phi_0^2}
\sin\left(k \pi \frac{\phi+ \phi_0}{2\phi_0}\right)\sin\left(k \pi \frac{\theta+ \phi_0}{2\phi_0}\right)\dd \theta,
\label{thetakilled}
\end{split}
\end{equation}
for $\vartheta\in(-\phi_0,\phi_0)$. Hence, in \eqref{needsbmsg}, we get 
\begin{align}
&\int_0^\infty\int_0^\infty \int_{(-\phi_0,\phi_0)} f(s, y, \theta)U^{-,\dagger}_\phi(\dd s, \dd y, \dd \theta) \notag\\
&=\frac{1}{\phi_0}\sum_{k =1}^\infty\int_0^\infty \dd y \int_{(-\phi_0,\phi_0)}
\mathbf{E}_{0,\phi}\left[f(\underline{L}^{-1}_y,y, \theta){\rm e}^{- \pi^2 k^2 \underline{L}^{-1}_y/8\phi_0^2}\right]\\
&\qquad\times\sin\left(k \pi \frac{\phi+ \phi_0}{2\phi_0}\right)\sin\left(k \pi \frac{\theta+ \phi_0}{2\phi_0}\right)\dd \theta.
\label{needsbmsg2}
\end{align}
Recalling that $\underline{L}^{-1}$ is a 1/2-stable subordinator, so that 
\begin{equation*}
\mathbf{P}_{0,\phi}(\underline{L}^{-1}_y \in \dd s) = \frac{y}{\sqrt{2\pi}}{\rm e}^{-y^2/2s}s^{-3/2}\dd s,\quad  s\geq0,
\end{equation*}
and
\begin{equation*}
\mathbf{E}_{0,\phi}[{\rm e}^{-q\underline{L}^{-1}_y}] ={\rm e}^{-y\sqrt{2q}}, \quad q,y\geq 0,
\label{IG}
\end{equation*} 
we have from \eqref{needsbmsg2} that 
\begin{align}
&\int_0^\infty\int_0^\infty \int_{(-\phi_0,\phi_0)} f(s, y, \theta)U^{-,\dagger}_\phi(\dd s, \dd y, \dd \theta) \notag\\
&=\frac{1}{\phi_0}\sum_{k =1}^\infty\int_0^\infty\int_0^\infty\int_{(-\phi_0,\phi_0)}f(s, y, \theta)\notag\\
&\hspace{2cm}
\frac{y}{\sqrt{2\pi}}{\rm e}^{-y^2/2s}s^{-3/2}
{\rm e}^{- \pi^2 k^2 s/8\phi_0^2}
\sin\left(k \pi \frac{\phi+ \phi_0}{2\phi_0}\right)\sin\left(k \pi \frac{\theta+ \phi_0}{2\phi_0}\right) \dd y \dd s \dd \theta.
\label{needsbmsg2}
\end{align}

From \eqref{needsbmsg2}
 we conclude that $ U^{-,\dagger}_\phi( \dd s, \dd y, \dd \theta)$ has a density with respect to $\dd s\dd y\dd\theta$ on $[0,\infty)^2\times (-\phi_0,\phi_0)$, say $ u^{-,\dagger}_\phi( s, y,  \theta)$,  which satisfies
 \begin{equation}
 u^{-,\dagger}_\phi( s, y,  \theta) = \frac{y}{\phi_0\sqrt{2\pi}} \frac{{\rm e}^{-y^2/2s}}{s^{3/2}}\sum_{k =1}^\infty 
{\rm e}^{- \pi^2 k^2 s/8\phi_0^2}
\sin\left(k \pi \frac{\phi+ \phi_0}{2\phi_0}\right)\sin\left(k \pi \frac{\theta+ \phi_0}{2\phi_0}\right) \dd y \dd s \dd \theta. 
 \label{ugammaequals}
\end{equation}

Embedded in our calculations, and for future reference, we note that 
\begin{equation}
 u^{-,\dagger}_\phi( s, y,  \theta)\dd s \dd y\dd\theta= \mathbf{P}_{0,\phi}(\underline{L}^{-1}_y\in \dd s)\mathbf{P}_{0,\phi} (\theta_s\in \dd \theta, \, s<\tau^{(-\phi_0,\phi_0)})\dd y,
\label{udaggerboring}
\end{equation}
on $[0,\infty)^2\times (-\phi_0,\phi_0)$.

Paraphrasing \cite{BS97}, there exists a ground state eigenfunction $M>0$ on $(-\phi_0,\phi_0)$ satisfying
\begin{equation}
\frac{\dd^2}{\dd \phi^2}M(\phi) =- \lambda_1 M(\phi) \text{ in }D,  \text{ and } M(\phi)=0 \text{ for  }\phi =\pm\phi_0,
\label{explicit}
\end{equation}
for some $\lambda_1>0$. Note that, as it is defined, the function $M$ is determined up to a multiplicative constant. Moreover, for each fixed $t>0$, up to the right scaling in $M$,
\begin{equation}
\mathbb{P}_x(\tau^\Gamma>t) \sim M(\arg(x))|x|^{\lambda_1^{1/2}}t^{-\lambda_1^{1/2}/2}, \qquad |x|\to\infty,
\label{preconditioned}
\end{equation}
where we understand $\arg(x) = \phi$ if we write $x = |x|{\rm e}^{{\rm i}\phi}$. 
Because we are in the special case of planar Brownian motion, \eqref{explicit} can be solved explicitly. Indeed, it is straightforward to see that 
\begin{equation}
M(\phi) = \sin\left(\pi \frac{\phi+ \phi_0}{2\phi_0}\right) \text{ and }\lambda_1 = \frac{\pi^2}{4\phi_0^2}.
\label{explM}
\end{equation}

Using \eqref{preconditioned} and \eqref{explM}, we can defined the Brownian motion conditioned to remain in $\Gamma$ via the limit 
\[
\mathbb{P}^\Gamma_x(A): = \lim_{s\to\infty} \mathbb{P}_x(A, \, t<T^\Gamma| \, t+s < T^\Gamma), \qquad A\in \mathcal{F}_t, t>0.
\]
It is not difficult to show using Bayes formula, the Markov property and dominated convergence that 
\begin{equation}
\left.\frac{\dd \mathbb{P}^\Gamma_x}{\dd \mathbb{P}_x}\right|_{\mathcal{F}_t} = 
\frac{
\sin\left(\pi (\arg(X_t)+ \phi_0)/2\phi_0\right)
}{\sin\left(\pi( \arg(x)+ \phi_0)/2\phi_0\right)}\left(\frac{|X_t|}{|x|}\right)^{\pi/2 \phi_0}\one{t<T^\Gamma}, \qquad t>0, x\in \Gamma.
\label{coneCOM}
\end{equation}

Note that 
\begin{equation}
{\rm e}^{(\pi/2 \phi_0)\xi_{t} -(\pi^2/8\phi_0^2)t} , \qquad t\geq 0,
\label{mg1}
\end{equation}
is a $\mathbb{P}$-martingale {\color{black} with respect to the process $\xi$, and because of the independence it is a martingale with respect to the couple $(\xi,\theta)$, as is the process 
\begin{equation}
{\rm e}^{(\pi^2 /8 \phi_0^2)t }\frac{\sin\left(\pi (\theta_t+ \phi_0)/2\phi_0\right)}{\sin\left(\pi (\phi+ \phi_0)/2\phi_0\right)}\one{t<\tau^{(-\phi_0,\phi_0)} }, \qquad t\geq0,
\label{mg2}
\end{equation}
with respect of of the law of the process $\theta$ issued from $\theta_0 = \phi$, see e.g. \cite{Lambert2000} page 256.
Since $\varphi(t)$ is a $\mathbb{P}$-almost surely finite stopping time, it follows that  the product of \eqref{mg1} and \eqref{mg2} albeit sampled at $\varphi(t)$  constitutes a change of measure on the pair $(\xi,\theta)$, which coincides with the one in \eqref{coneCOM}}.  Indeed, \eqref{coneCOM}, has the effect of conditioning $(\theta_t, t\geq0)$ to remain in $(-\theta_0,\theta_0)$, the so called taboo process, as well as independently adding a drift of magnitude $\pi/2 \phi_0$ to $(\xi_t,\geq0)$.   The combined martingale change of measure as seen from the perspective  $((\xi,\theta),\mathbf{P})$ is written
\begin{equation}\label{70}
\left.\frac{\dd \mathbf{P}^\Gamma_{0,\phi}}{\dd \mathbf{P}_{0,\phi}}\right|_{\sigma((\xi_s, \theta_s), s\leq t)}=
{\rm e}^{(\pi/2 \phi_0)\xi_{t}} 
\frac{\sin\left(\pi (\theta_t+ \phi_0)/2\phi_0\right)}{\sin\left(\pi (\phi+ \phi_0)/2\phi_0\right)}\one{t<\tau^{(-\phi_0,\phi_0)} }, \qquad t\geq 0.
\end{equation}
In turn, remembering that $H^-_y: = -\xi_{\underline{L}^{-1}_y} = y$ this culminates in the descending ladder potential, $u^-(s,y,\theta)$, $s,y\geq0$, $\theta\in (-\phi_0,\phi_0)$, taking the form
\begin{align}
u^{-}_\phi( s, y,  \theta) \dd s\dd y\dd\theta& = 
{\rm e}^{-(\pi/2 \phi_0)y - (\pi^2/8\phi_0^2)s} \times
{\rm e}^{(\pi^2/8\phi_0^2)s}\frac{\sin\left(\pi (\theta+ \phi_0)/2\phi_0\right)}{\sin\left(\pi (\phi+ \phi_0)/2\phi_0\right)}u^{-,\dagger}_\phi( s, y,  \theta) \dd s\dd y\dd\theta\notag\\
&={\rm e}^{-(\pi/2 \phi_0)y - (\pi^2/8\phi_0^2)s}\mathbf{P}_{0,\phi}(\underline{L}^{-1}_y\in \dd s)\notag\\
&\hspace{1cm}\times 
{\rm e}^{(\pi^2/8\phi_0^2)s}\frac{\sin\left(\pi (\theta+ \phi_0)/2\phi_0\right)}{\sin\left(\pi (\phi+ \phi_0)/2\phi_0\right)} 
\mathbf{P}_{0,\phi} (\theta_s\in \dd \theta, \, s<\tau^{(-\phi_0,\phi_0)}) \dd y
\notag\\
&=\mathbf{P}_{0,\phi}^\Gamma\left( \underline{L}^{-1}_y\in \dd s, \, y<\zeta\right) \mathbf{P}_{0,\phi}^\Gamma\left( \theta^-_y\in \dd \theta\right)\dd y,
\label{ass8holds}
\end{align}
where necessarily, 
\[
\mathbf{P}_{0,\phi}^\Gamma\left( \theta^-_y\in \dd \theta\right) = {\rm e}^{(\pi^2/8\phi_0^2)s}\frac{\sin\left(\pi (\theta+ \phi_0)/2\phi_0\right)}{\sin\left(\pi (\phi+ \phi_0)/2\phi_0\right)} \mathbf{P}_{0,\phi}^\Gamma\left( \theta^-_y\in \dd \theta\right),
\]
and 
\begin{equation}
\mathbf{P}_{0,\phi}^\Gamma\left( \underline{L}^{-1}_y\in \dd s, \, y<\zeta\right) ={\rm e}^{-(\pi/2 \phi_0)y - (\pi^2/8\phi_0^2)s}\mathbf{P}_{0,\phi}(\underline{L}^{-1}_y\in \dd s).
\label{problem}
\end{equation}
On the one hand,  under $\mathbf{P}^\Gamma$, the resultant change of measure \eqref{mg1} tells us that $\xi$ has Laplace exponent $\psi_\xi(\lambda) = \lambda^2/2 + \pi\lambda/2 \phi_0 \propto \lambda (\lambda+ \pi/\phi_0)$. This suggests that $H^-$ and hence 
$\underline{L}^{-1}$ is killed at rate $\mathtt{n}^{-}(\zeta = \infty) = \pi/ \phi_0$ under $\mathbf{P}^\Gamma$. 
We can also see this by integrating \eqref{problem} over $s\geq0$. In particular, appealing to \eqref{IG}, we have, for $y\geq0$,
\begin{equation*}
\begin{split}
\mathbf{P}_{0,\phi}^\Gamma\left( y<\zeta\right)&=\int_{[0,\infty)}{\rm e}^{-(\pi/2 \phi_0)y - (\pi^2/8\phi_0^2)s}\mathbf{P}_{0,\phi}(\underline{L}^{-1}_y\in \dd s)\\
&={\rm e}^{-(\pi/2 \phi_0)y-y\sqrt{2(\pi^2/8\phi_0^2)}}={\rm e}^{-y(\pi/ \phi_0)}.
\end{split}
\end{equation*}

Putting the pieces together and integrating over $s$ using \eqref{IG}, we have 
\begin{equation}
 u^{-}_\phi(y,  \theta) = \frac{1}{\phi_0} \sum_{k =1}^\infty 
{\rm e}^{- \pi (k+1) y/2\phi_0}
\sin\left(k \pi \frac{\phi+ \phi_0}{2\phi_0}\right)\sin\left(k \pi \frac{\theta+ \phi_0}{2\phi_0}\right).
\label{uminusfinal}
\end{equation}

Next we verify that Assumptions \ref{ass1}, \ref{ass2}, \ref{ass3}, \ref{ass4}, \ref{ass5}, \ref{ass6} and \ref{ass8} (which implies \ref{ass7}). Assumption \ref{ass1} is clear, indeed $(\theta_t, t\geq0)$ under $\mathbb{P}^\Gamma$ is the so-called and well-studied `taboo process'. For Assumption \ref{ass2}, we note that $\pi(\dd\theta)$ is nothing more than the invariant distribution for the taboo process, which is well known to be proportional to $\sin^2\left(\pi (\phi+ \phi_0)/2\phi_0\right)$, $\phi\in(-\phi_0,\phi_0)$, see e.g. \cite{Lambert2000}. The required duality is an immediate consequence of the fact that $\xi$ and $\theta$ are independent and $\theta$ is {\color{black} dual to Brownian motion killed at the first exist from the interval $(-\phi,\phi),$ with respect to $\pi$, which is enough for our purposes.
Assumptions \ref{ass3} and \ref{ass4} are immediate by virtue of the fact that $\xi$ is a Brownian motion with positive drift under $\mathbb{P}^\Gamma$.} {The dual to $\xi$ is Brownian motion with drift $-\pi/2 \phi_0$. Hence, under $\mathbb{P}^\Gamma$, the associated inverse local time at the minimum is  a tempered stable subordinator  and necessarily independent of the dual to $\theta$. This is easily deduced by noting the inverse local time is also a first passage time, i.e. a stopping time, and a Brownian motion with negative drift can be related to a standard Brownian motion via the Girsanov transform; this in turn converts stable-distributed first passage times to tempered stable-distributed first passage times. The  inverse local time at the minimum of the dual of $\xi$ being  a tempered stable subordinator thus  implies that the condition in Assumption \ref{ass5} is also satisfied under $\mathbb{P}^\Gamma$. Assumption \ref{ass6} is trivially satisfied, thanks again to the independence of $\xi$ and $\theta$ and that they are both processes with transition densities.
Finally 
Assumption \ref{ass8} holds automatically on account of the density \eqref{ass8holds}.
 }

\smallskip

%
%
%
%
%

We can now read out of Theorem \ref{Wthrm} that the path of planar Brownian motion conditioned to remain in the cone $\Gamma$ is deconstructed as follows. 
\smallskip

The point of closest reach to the apex of the cone is determined by the variable $X^*$ such that, writing $X^* = \exp(-y^*+ {\rm i}\theta^*)$, then 
\[
\mathbb{P}^\Gamma_x(y^*\in \dd y,\, \theta^*\in \dd \theta) = \frac{1}{\phi_0} \sum_{k =1}^\infty 
{\rm e}^{- \pi (k+1) y/2\phi_0}
\sin\left(k \pi \frac{\phi+ \phi_0}{2\phi_0}\right)\sin\left(k \pi \frac{\theta+ \arg(x)}{2\phi_0}\right) \dd y  \dd \theta,
\]
From \eqref{Hdown}, 
we have 
 \[
H^\downarrow(x{\rm e}^{-y}; \{\theta\}) = u^-_{\arg(x)}(\log||x||-y, \theta) .
\]
Hence, from Theorem \ref{Wthrm}, on the event $\{y^*\in \dd y,\, \theta^*\in \dd \theta\}$, with $y<||x||$, we have, for example, that the path preceding the radial minimum (or point of closest reach to the origin)
is given by the change of measure
\begin{equation*}
\begin{split}
	\left.\frac{\dd \mathbb{P}^{\downarrow, \Gamma, (y)}_x}{\dd \mathbb{P}_x}\right|_{\mathcal{F}_t \cap \{ t < \zeta \}} &= \one{t<T^\Gamma \wedge T_{\exp(y)}}\frac{
\sin\left(\pi (\arg(X_t)+ \phi_0)/2\phi_0\right)
}{\sin\left(\pi( \arg(x)+ \phi_0)/2\phi_0\right)}\left(\frac{|X_t|}{|x|}\right)^{\pi/2 \phi_0}\notag\\
&\times \frac{ \sum_{k =1}^\infty 
{\rm e}^{- \pi (k+1) (\log||X_t||-y)/2\phi_0}
\sin\left(k \pi \frac{\arg(X_t)+ \phi_0}{2\phi_0}\right)\sin\left(k \pi \frac{\theta+ \phi_0}{2\phi_0}\right) }{ \sum_{k =1}^\infty 
{\rm e}^{- \pi (k+1) (\log ||x|| -y)/2\phi_0}
\sin\left(k \pi \frac{\arg(x)+ \phi_0}{2\phi_0}\right)\sin\left(k \pi \frac{\theta+ \phi_0}{2\phi_0}\right) }, 
\end{split}
\end{equation*}
for $t>0$ and $x\in \Gamma$.

\section*{Acknowledgements}
AEK acknowledges support from the EPSRC grant EP/S036202/1. MM was supported by a scholarship from the EPSRC Centre for Doctoral Training in Statistical Applied Mathematics at Bath (SAMBa) under the project code EP/S022945/1. This paper was concluded while VR was visiting the Department of Statistics at the University of Warwick, United Kingdom; he would like to thank his hosts for partial financial support as well as for their kindness and hospitality. In addition, VR is grateful for additional financial support  from  CONAHCyT-Mexico, grant nr. 852367.


\begin{thebibliography}{10}

\bibitem{AlsmeyerTheMR}
G.~Alsmeyer.
\newblock The {M}arkov renewal theorem and related results.
\newblock {\em Markov Process. Related Fields}, 3(1):103--127, 1997.

\bibitem{AsmussenQueue}
S.~Asmussen.
\newblock {\em Applied probability and queues}, volume~51 of {\em Applications
  of Mathematics (New York)}.
\newblock Springer-Verlag, New York, second edition, 2003.
\newblock Stochastic Modelling and Applied Probability.

\bibitem{AA}
S.~Asmussen and H.~Albrecher.
\newblock {\em Ruin probabilities}, volume~14 of {\em Advanced Series on
  Statistical Science \& Applied Probability}.
\newblock World Scientific Publishing Co. Pte. Ltd., Hackensack, NJ, second
  edition, 2010.

\bibitem{BS97}
Rodrigo Ba\~{n}uelos and Robert~G. Smits.
\newblock Brownian motion in cones.
\newblock {\em Probab. Theory Related Fields}, 108(3):299--319, 1997.

\bibitem{bertoin1998cambridge}
J.~Bertoin.
\newblock {\em Cambridge Tracts in Mathematics}.
\newblock Number v. 121 in Cambridge Tracts in Mathematics. Cambridge
  University Press, 1998.

\bibitem{BlumenthalGetoor}
R.M.C. Blumenthal and R.K. Getoor.
\newblock {\em Markov Processes and Potential Theory}.
\newblock Dover books on mathematics. Dover Publications, 2007.

\bibitem{CinlarMRT}
E.~\c{C}inlar.
\newblock Markov renewal theory.
\newblock {\em Advances in Appl. Probability}, 1:123--187, 1969.

\bibitem{CinlarI&II}
E.~\c{C}inlar.
\newblock Markov additive processes. {I}, {II}.
\newblock {\em Z. Wahrscheinlichkeitstheorie und Verw. Gebiete}, 24:85--93;
  ibid. 24 (1972), 95--121, 1972.

\bibitem{Cinlar1}
E.~\c{C}inlar.
\newblock L\'{e}vy systems of {M}arkov additive processes.
\newblock {\em Z. Wahrscheinlichkeitstheorie und Verw. Gebiete}, 31:175--185,
  1974/75.

\bibitem{CPR}
L.~Chaumont, H.~Pant\'{\i}, and V.~Rivero.
\newblock The {L}amperti representation of real-valued self-similar {M}arkov
  processes.
\newblock {\em Bernoulli}, 19(5B):2494--2523, 2013.

\bibitem{chaumont-rivero}
L.~Chaumont and V.~Rivero.
\newblock On some transformations between positive self-similar {M}arkov
  processes.
\newblock {\em Stochastic Process. Appl.}, 117(12):1889--1909, 2007.

\bibitem{jacod2013limit}
J.~Jacod and A.N. Shiryaev.
\newblock {\em Limit Theorems for Stochastic Processes}.
\newblock Grundlehren der mathematischen Wissenschaften. Springer Berlin
  Heidelberg, 2013.

\bibitem{Kaspi2}
H.~Kaspi.
\newblock On the symmetric {W}iener-{H}opf factorization for {M}arkov additive
  processes.
\newblock {\em Z. Wahrsch. Verw. Gebiete}, 59(2):179--196, 1982.

\bibitem{Kaspi1}
H.~Kaspi.
\newblock Excursions of {M}arkov processes: an approach via {M}arkov additive
  processes.
\newblock {\em Z. Wahrsch. Verw. Gebiete}, 64(2):251--268, 1983.

\bibitem{kesten69}
H.~Kesten.
\newblock {\em Hitting probabilities of single points for processes with
  stationary independent increments}.
\newblock Memoirs of the American Mathematical Society, No. 93. American
  Mathematical Society, Providence, R.I., 1969.

\bibitem{KPal}
P.~Klusik and Z.~Palmowski.
\newblock A note on {W}iener-{H}opf factorization for {M}arkov additive
  processes.
\newblock {\em J. Theoret. Probab.}, 27(1):202--219, 2014.

\bibitem{KPS}
A.~E. Kyprianou, S.~Palau, and Ts. Saizmaa.
\newblock Attraction to and repulsion from a subset of the unit sphere for
  isotropic stable {L}\'{e}vy processes.
\newblock {\em Stochastic Process. Appl.}, 137:272--293, 2021.

\bibitem{ssmpbook}
A.~E. Kyprianou and J.~C. Pardo.
\newblock {\em {S}table {L}{é}vy {P}rocesses via {L}amperti-{T}ype
  {R}epresentations}.
\newblock Institute of Mathematical Statistics Monographs. Cambridge University
  Press, 2022.

\bibitem{kyprianou2018deep}
A.~E. Kyprianou, V.~Rivero, and W.~Satitkanitkul.
\newblock Deep factorisation of the stable process iii: Radial excursion theory
  and the point of closest reach, 2018.

\bibitem{TAMS}
A.~E. Kyprianou, V.~Rivero, B.~Şengül, and T.~Yang.
\newblock Entrance laws at the origin of self-similar markov processes in high
  dimensions, 2019.

\bibitem{Lambert2000}
A.~Lambert.
\newblock Completely asymmetric {L}\'{e}vy processes confined in a finite
  interval.
\newblock {\em Ann. Inst. H. Poincar\'{e} Probab. Statist.}, 36(2):251--274,
  2000.

\bibitem{maisonneuve}
B.~Maisonneuve.
\newblock {Exit Systems}.
\newblock {\em The Annals of Probability}, 3(3):399 -- 411, 1975.

\bibitem{millar1978}
P.~W. Millar.
\newblock {A Path Decomposition for Markov Processes}.
\newblock {\em The Annals of Probability}, 6(2):345 -- 348, 1978.

\bibitem{RY}
D.~Revuz and M.~Yor.
\newblock {\em Continuous martingales and {B}rownian motion}, volume 293 of
  {\em Grundlehren der mathematischen Wissenschaften [Fundamental Principles of
  Mathematical Sciences]}.
\newblock Springer-Verlag, Berlin, third edition, 1999.

\bibitem{silverstein}
M.~L. Silverstein.
\newblock {Classification of Coharmonic and Coinvariant Functions for a Levy
  Process}.
\newblock {\em The Annals of Probability}, 8(3):539 -- 575, 1980.

\bibitem{Williams}
D.~Williams.
\newblock Path decomposition and continuity of local time for one-dimensional
  diffusions. {I}.
\newblock {\em Proc. London Math. Soc. (3)}, 28:738--768, 1974.

\end{thebibliography}
\end{document}